\documentclass[11pt]{amsart}
\usepackage[all,ps,tips,tpic]{xy}
\usepackage{tikz-cd}
\usepackage{graphicx}
\usepackage[tableposition=above]{caption}

\usepackage{amsmath,amscd,amssymb}
\usepackage{hyperref}
\newtheorem{theorem}{Theorem}[section]

\newtheorem{lemma}[theorem]{Lemma}
\newtheorem{proposition}[theorem]{Proposition}
\newtheorem{remark}[theorem]{Remark}
\numberwithin{equation}{section}
\newtheorem{definition}[theorem]{Definition}
\newtheorem{example}[theorem]{Example}
\newtheorem{corollary}[theorem]{Corollary}

\input xy
\xyoption{all}

\usepackage{mathrsfs}

\usepackage{caption}
\usepackage{csquotes}

\begin{document}
	\title[non-ACM line bundles over cubic surfaces]
	{$\ell$-away aCM line bundles on a nonsingular cubic surface}
	
	\author[D. Bhattacharya, A. J. Parameswaran and J. Pine]{Debojyoti Bhattacharya, A.J. Parameswaran and Jagadish Pine}
	\address{School of Mathematics, Tata Institute of Fundamental Research, Homi Bhabha
		Road, Colaba, Mumbai-400005, India}
	\email{debojyoti7054@gmail.com; debojyot@math.tifr.res.in}
	\address{Kerala School of Mathematics, Kunnamangalam PO, Kozhikode-673571, Kerala, India}
	\email{param@math.tifr.res.in}
	\address{School of Mathematics, Tata Institute of Fundamental Research, Homi Bhabha
		Road, Colaba, Mumbai-400005, India}
	\email{jagadish@math.tifr.res.in}


	\begin{abstract}
		Let $X \subset \mathbb P^3$ be a nonsingular cubic hypersurface. Faenzi (\cite{F}) and later Pons-Llopis and Tonini (\cite{PLT}) have completely characterized ACM line bundles over $X$. As a natural continuation of their study in the non-ACM direction, in this paper, we completely classify $\ell$-away ACM line bundles (introduced recently by Gawron and Genc (\cite{GG})) over $X$, when $\ell \leq 2$. For $\ell\geq 3$, we give examples of $\ell$-away ACM line bundles on $X$ and  for each  $\ell \geq 1$, we establish the existence of smooth hypersurfaces $X^{(d)}$ of degree $d >\ell$ in $\mathbb P^3$ admitting $\ell$-away ACM line bundles.\
	\end{abstract}
	
	\keywords{Cubic surface; Line bundle; $\ell$-away ACM bundle}
	
	\subjclass[2020]{Primary 14J60; Secondary 14J45 · 14F05}
	
	\maketitle
	
	\section{Introduction}
	A celebrated theorem due to Grothendieck states that every vector bundle on $\mathbb P^1$ splits into direct sum of line bundles (see \cite{okonek}, Theorem $2.1.1$ and section $2.4$ for historical remarks). This does not generalize to projective spaces of dimension at least $2$. Indeed a seminal classification result due to Horrocks states that a vector bundle $E$ over $\mathbb P^n$ ($n \geq 2$) splits if and only if its intermediate cohomologies vanish i.e. $H^i(\mathbb P^n, E(k))=0$ for all $k \in \mathbb Z$ and $0 <i<n$ (see \cite{okonek}, Theorem $2.3.1$). Although  Knörrer's theorem (see \cite{quadric} or \cite{book}, Theorem $2.3.4$) asserts that even this does not generalize to arbitrary smooth projective varieties (e.g. on smooth quadric hypersurfaces $Q_n \subset \mathbb P^{n+1}, n \geq 3$), Horrocks splitting criteria motivates the definition of ACM bundles as follows: Let $X$ be a smooth projective variety and $H$ be a very ample divisor on $X$ giving the embedding  $X \subset \mathbb P^n$. Then a vector bundle $E$ on $X$ is said to be ACM with respect to $H$ if $H^i(X, E \otimes \mathcal O_X(H)^{\otimes t}) =0$ for all $t \in \mathbb Z$ and $0 < i < \text{dim}(X)$. The folklore guiding philosophy here is that `simple' varieties should support `simple' category of ACM bundles. Consequently, in the past two decades, considerable efforts have been driven towards the study of ACM bundles and Ulrich bundles (which are ACM bundles whose minimal sets of generators have the largest possible cardinality) over a smooth polarized projective variety. See \cite{book}, \cite{Beauville}, \cite{Coskun} and \cite{story} for a comprehensive survey.\
	
	As a first step in the non-ACM direction, it is therefore natural to now turn our attention to the non-ACM bundles on  a smooth polarized projective variety in the sense of Gawron and Genc, (see \cite{GG}, definition $1.1$) defined as follows: Let $X$ be a smooth projective variety and $H$ be a very ample divisor on $X$ giving the embedding  $X \subset \mathbb P^n$. Then a vector bundle $E$ on $X$ is said to be $\ell$-away ACM with respect to $H$ if there are exactly $\ell$ pairs $(i,t) \in \mathbb Z^{2}$ with $0 <i <\text{dim}(X)$ such that $H^i(X, E \otimes \mathcal O_X(H)^{\otimes t}) \neq 0$. With this notation, let $S(E):=\{(i,t) \in \mathbb Z^2 \hspace{1.5mm} | \hspace{1.5mm} 0 <i <\text{dim}(X), H^i(X, E \otimes \mathcal O_X(H)^{\otimes t}) \neq 0  \}$ and $\ell (E):=\# S(E)$. By Serre vanishing and Serre duality, it follows that $S(E)$ is a finite set and thus every vector bundle $E$ on a smooth polarized projective variety is $\ell(E)$-away ACM for some non-negative integer $\ell(E)$. When $E$ is clear from the context, we simply write $\ell$-away ACM instead of $\ell(E)$-away ACM. For $\ell=0$, one recovers the definition of ACM bundles and the interest lies in the case of $\ell$-away ACM bundles with low values of $\ell$. In the present paper, we take up the classification problem  when $\ell=1,2$. Note that, besides being a natural question of study in its own right,  the so called weakly Ulrich (Definition \ref{Wul}) and supernatural bundles (see \cite{GG}, Definition $3.7$) also serve as  crucial motivating factors to study $\ell$-away ACM bundles with $\ell>0$, since these two type of bundles can posses non-vanishing intermediate cohomology groups.\
	
	In the light of the above definition, let us now consider the non-ACM version of Horrocks theorem. Let $E$ be a rank $r \geq 2$, $\ell (>0)$-away ACM bundle on $\mathbb P^n (n \geq 2)$. Then certainly $E$ does not split. Therefore, it is natural to ask for a precise classification of $E$. The case $n=2, r=2$  was studied by Gawron and Genc in \cite{GG}. In our upcoming paper, in \cite{BPP}, we have considered the case of such bundles of higher rank on $\mathbb P^2$. We mention that, in \cite{GG}, the authors also studied $\ell (>0)$-away ACM bundles on certain Fano surfaces other than $\mathbb P^2$.  \
	
	Now we focus our attention to the rank $1$ i.e. line bundle case in this context. If an ACM bundle  splits into a direct sum of line bundles, then the line bundles  appearing in the splitting are also ACM and conversely, extension of ACM line bundles are also ACM. Hence, the classification of ACM line bundles on a smooth projective variety is important. In literature, one can find several results on ACM line bundles over a fixed surface. For instance, when the surface is Del-Pezzo, Pons-Llopis and Tonini have classified ACM line bundles on them in terms of rational normal curves that they contain (see \cite{PLT}). Chindea has classified ACM line bundles over a complex polarized elliptic ruled surface (see \cite{Chindea}).  ACM bundles on a smooth quadric surface were completely classified by Knörrer (see \cite{k}). Faenzi classified ACM line bundles on a smooth cubic surface in $\mathbb P^3$ (see \cite{F}). Watanabe has classified ACM line bundles over a smooth quartic hypersurface in $\mathbb P^3$ (see \cite{quartic}), $K3$ surface (see \cite{K3}) and quintic surface in $\mathbb P^3$ (see \cite{Watanabe}). Bhattacharya (see \cite{Debu}) has studied ACM line bundles over hypersurfaces of degree $d \geq 6$ in $\mathbb P^3$ and obtained the classfication for $d=6$.\
	
	Note that in the non-ACM context, if an $\ell$-away ACM bundle decomposes into smaller rank bundles then each of the bundle appearing in the decomposition is atmost $\ell$-away ACM. Conversely, let $E_1$, $E_2$ be two vector bundles on a smooth polarized projective variety such that $E_1$ is $\ell(E_1)$-away ACM and $E_2$ is $\ell(E_2)$-away ACM. Then $E=E_1 \oplus E_2$ is $\ell(E)$ away ACM, where $\ell(E):=\ell(E_1) +\ell(E_2) - \# (S(E_1) \cap S(E_2))$. Therefore, a characterization of $\ell$-away ACM line bundle on a fixed polarized smooth projective variety is the starting point. Towards this, Gawron and Genc have completely classified $\ell$-away ACM line bundles on a smooth quadric hypersurface in $\mathbb P^3$ using Künneth’s theorem (see \cite{GG}, Theorem $4.1$) and in the same paper the authors have also studied $\ell$-away ACM line bundles on $\mathbb P^2$ blown up at atmost $3$  non-collinear points.  From \cite{GG}, it is known that the unique $1$-away ACM bundle of rank $2$ on $\mathbb{P}^2$ is the pushforward of a $1$-away ACM line bundle on a smooth quadric hypersurface $Q \subset \mathbb{P}^3$ viewed as a degree $2$ covering of $\mathbb{P}^2$. In the $2$-away case, the situation is different: there are two $2$-away ACM line bundles on $Q$, whereas in our upcoming paper (see \cite{BPP}), we have shown that $2$-away ACM bundles of rank $2$ on $\mathbb P^2$ varies in a $3$ dimensional family in the first case of (see \cite{GG}, Theorem $3.3$). 
	Our motivation is to extend this study to $\ell$-away ACM bundles of higher ranks on $\mathbb{P}^2$. Thus, in this context, it is natural to consider $\ell$-away ACM line bundles on a smooth cubic hypersurface $X \subseteq \mathbb{P}^3$ viewed as a degree $3$ covering of $\mathbb{P}^2$. In this article, we will prove classification results of $1$-away and $2$-away ACM line bundles on $X$ and provide examples of $\ell$-away ACM line bundles for $\ell >2$. In addition, we will also establish the existence of higher degree smooth hypersurfaces in $\mathbb P^3$ admitting $\ell$-away ACM line bundles. 
	 This paper can also be considered as a contribution in the series of papers in the literature on several aspects of  line bundles over a smooth cubic surface (e.g see \cite{88}, \cite{90}, \cite{92ii}, \cite{92}, \cite{96}, \cite{F}, \cite{PLT} etc.)\ 
	 
	
	
	\subsection*{Main results and plan of the paper}\
	
	 In section $\S \ref{prelim}$,  we collect few definitions and facts related to the basic objects concerned which will be required in the subsequent sections.\
	 
	 In section $\S \ref{tech}$, we prepare some technical results on the properties of certain line bundles over a cubic surface which are useful for the proofs of the  main theorems of this article in  sections $\S\ref{cl12}$ and $\S\ref{ex}$. For convenience, we divide the later part of this section into two subsections. In the first subsection i.e. in  $\S\ref{DB}$, we establish a necessary degree bound result for $\ell(\geq 2)$-away ACM line bundles on $X$ (see Proposition \ref{degreebound}). In the second subsection i.e in  $\S\ref{12}$, we motivate the studies carried out in the later sections by documenting examples of $\ell$-away ACM line bundles of low degree  (degree $1,2$) on $X$.\ 

	
	  In section $\S \ref{cl12}$, in the first subsection $\S \ref{cl1}$, we  classify  initialized (see Definition \ref{In})  and $1$-away ACM line bundles on a smooth cubic surface as follows:\
	
	\begin{theorem}(= Theorem \ref{T31})
			Let $D$ be a non-zero effective divisor on a nonsingular cubic hypersurface $X \subset \mathbb P^3$. Then the following conditions are equivalent: 
		
		$(a)$ $\mathcal O_X(D)$ is initialized  and $1$-away ACM with respect to $H =\mathcal O_{\mathbb P^3}(1)|_X$.\
		
		$(b)$ either $(D^2, D.H)=(-2,2)$ or $(D^2, D.H)=(2,4)$.
	\end{theorem}

In addition, we also give the explicit list of such divisors in Proposition \ref{P32}. In the second subsection $\S \ref{cl2}$,  we obtain a complete classification of initialized and $2$-away ACM line bundles on a smooth cubic surface as follows:\



\begin{theorem}(= Theorem \ref{T35})
	Let $D$ be a non-zero effective divisor on a nonsingular cubic hypersurface $X \subset \mathbb P^3$. Then the following conditions are equivalent: 
	
	$(a)$ $\mathcal O_X(D)$ is initialized  and $2$-away ACM with respect to $H =\mathcal O_{\mathbb P^3}(1)|_X$.\.\
	
	$(b)$ One of the following cases occur:
	
	$(i)$ $(D^2, D.H)=(4,6)$,  $D$ is nef and $|3H-D|=\emptyset$.\
	
	$(ii)$ $(D^2, D.H)=(3,5)$, $D$ is nef.\
	
	$(iii)$ $(D^2, D.H)=(0,4)$, $|2H-D|=\emptyset$. \
	
	$(iv)$ $(D^2, D.H)=(-3,3)$.\
	
	$(v)$ $(D^2, D.H)=(-1,3)$.\

\end{theorem}

We then give the explicit list of such divisors in Proposition \ref{P37} and as a consequence of the above two theorems, in Corollary \ref{notUlrich}, we also classify weakly Ulrich line bundles on $X$ that are not Ulrich.\

In section $\S \ref{ex}$, we give examples of $\ell$-away ACM line bundles on a nonsingular cubic surface. In the first example (see \ref{ex1}), we determine the values of $\ell$ for which a non-zero effective divisor $D$ on a smooth cubic surface $X$ with $D.H=3$ becomes $\ell$-away ACM and thereby complete the discussion on $\ell$-away ACM line bundles for line bundles of degree upto $3$ on $X$.\

 Next, let us consider a cubic surface $X$ being isomorphic with $\mathbb{P}^2$ blown up at $6$ points $P_1, P_2, \cdots, P_6$ in general  position. Let $\pi:X \rightarrow \mathbb{P}^2$ be the blow up map. Then any integral divisor in $X$ can be written as $\alpha l + \sum_{i=1}^6\beta_ie_i$, where $\alpha, \beta_i$ are integers, and $l$ is the class of pullback of a line in $\mathbb{P}^2$, and $e_i$ are $6$ exceptional divisors. With this notation, we then produce, examples of $\ell$-away ACM line bundles on $X$ for all natural numbers $\ell$ except the odd multiples of $3$ (i.e., for all $\ell \in \mathbb{N} \setminus 3(2\mathbb{N}+1)$) as in the following theorem:


\begin{theorem}(= Theorem \ref{Ext})
	Let $a \ge 0$ be an integer. Then we have the following:
	\begin{itemize}
		\item [(i)] The divisor $D=2l+a\sum_{i=1}^{6}e_{i}$ is $(6a+2)$-away ACM. The nonvanishing cohomologies are $H^{1}(\mathcal O_X(D-(5a+3)H)), \cdots, H^{1}(\mathcal O_X(D+(a-2)H))$.\ 
		
		\item[(ii)] The divisor $D=2l+(a+1)e_{1}+a\sum_{i=2}^{6}e_{i}$ is $(6a+4)$-away ACM. The nonvanishing cohomologies are $H^{1}(\mathcal O_X(D-(5a+4)H)), \cdots, H^{1}(\mathcal O_X(D+(a-1)H))$.\

		\item [(iii)] The divisor $D=2l+(a+1)e_{1}+(a+1)e_{2}+a\sum_{i=3}^{6}e_{i}$ is $(6a+5)$-away ACM. The nonvanishing cohomologies are $H^{1}(\mathcal O_X(D-(5a+5)H)), \cdots, H^{1}(\mathcal O_X(D+(a-1)H))$.\

		\item [(iv)] The divisor $D=2l+(a+1)e_{1}+(a+1)e_{2}+(a+1)e_{3}+a\sum_{i=4}^{6}e_{i}$ is $(6a+6)$-away ACM. The nonvanishing cohomologies are $H^{1}(\mathcal O_X(D-(5a+6)H)), \cdots, H^{1}(\mathcal O_X(D+(a-1)H))$.  \

		\item [(v)] The divisor $D=2l+(a+1)e_{1}+(a+1)e_{2}+(a+1)e_{3}+(a+1)e_{4}+a\sum_{i=5}^{6}e_{i}$ is $(6a+7)$-away ACM. The nonvanishing cohomologies are $H^{1}(\mathcal O_X(D-(5a+7)H)), \cdots, H^{1}(\mathcal O_X(D+(a-1)H))$. \

		\item [(vi)] The divisor $D=2l+(a+1)e_{1}+(a+1)e_{2}+(a+1)e_{3}+(a+1)e_{4}+(a+1)e_{5}+ae_{6}$ is $(6a+8)$-away ACM. The nonvanishing cohomologies are $H^{1}(\mathcal O_X(D-(5a+8)H)), \cdots, H^{1}(\mathcal O_X(D+(a-1)H))$. \

	\end{itemize}
\end{theorem}

 Finally, in section $\S \ref{hyp}$, we establish the following theorem, which guarantees the existence of higher degree smooth hypersurfaces in $\mathbb P^3$ admitting initialized and $\ell$-away ACM line bundles.\

\begin{theorem}(=Theorem \ref{T41})
	Let $\ell \geq 1$ be an integer. Then for any $d> \ell$, there exists a smooth hypersurface $X^{(d)} \subset \mathbb P^3$ of degree $d$  admitting  initialized and $\ell$-away ACM line bundles.
	
\end{theorem}


\section*{Notation and Convention}
\begin{itemize}
	\item We work over the field of complex numbers.\

	\item 	Let $X$ be a smooth projective variety and $H$ be a very ample divisor on $X$ giving the embedding $X \subset \mathbb P^n$. For a vector bundle $E$ on $X$, $E(t)$ for some $t \in \mathbb Z$ means $E \otimes \mathcal O_X(H)^{\otimes t}$.\
	
	\item For a vector bundle $E$ over a smooth projective variety $X$, $H^i(X, E)$ stands for the i'th cohomology group of $E$ and $h^i(X, E):= \text{dim}_{\mathbb C}(H^i(X, E))$. We will simply write $H^i(E)$ (or $h^i(E)$) when $X$ is clear from the context.\
	
	\item Let $m \in \mathbb N$. Then the phrase `$E$ is $\ell (\geq m)$-away ACM' means that $E$ is $\ell$-away ACM where $\ell \geq m$.\
	
	\item For a line bundle $\mathcal L$ on a smooth projective variety, $\mathcal L^*$ stands for its dual.\
	
	\item We freely interchange between a non-zero effective divisor $D$ and its associated line bundle $\mathcal O_X(D)$ on a smooth projective variety $X$.\
	
	\item For a non-zero effective divisor $D$ on a smooth projective surface $X$, we mean by $p_a(D)$ the arithmetic genus of $D$ given by $p_a(D) =\frac{D(D+K_X)}{2}+1$, where $K_X$ is the canonical divisor of $X$.\

	\item In sections $\S \ref{tech}, \S \ref{cl12}$ and $\S \ref{ex}$, $X$ will always stand for a smooth cubic hypersurface in $\mathbb P^3$ and in section  $\S \ref{hyp}$, $X^{(d)}$ will stand for a smooth hypersurface of degree $d$ in $\mathbb P^3$.\
	
	\item In section  $\S \ref{hyp}$, for a divisor $D$ on a smooth hypersurface $X^{(d)} \subset \mathbb P^3$, $\mathcal I_D:= \mathcal I_{D, \hspace{0.5mm} \mathbb P^3}$ i.e. ideal sheaf of $D$ on $\mathbb P^3$.\

	
	
	\item By S.E.S, we mean a short exact sequence.\
\end{itemize}

	\section{Preliminaries}\label{prelim}
	In this section, we  recall the basic definitions and facts related to the main objects concerned, e.g. cubic surface, nef divisors on cubic surface, ACM bundles, $\ell$-away ACM bundles, weakly Ulrich bundles etc.\\

	
	
	 Note that a nonsingular cubic hypersurface $X \subset \mathbb P^3$ is isomorphic to $\mathbb P^2$ blown up at $6$ points $P_1, P_2, \cdots, P_6$ in general  position. Let $\pi: X \to \mathbb P^2$ be the projection. Let $E_1,...,E_6 \subset X$  be the exceptional curves, and
	let $e_1,...e_6 \in \text{Pic}(X)$ be their linear equivalence classes. Let $l \in \text{Pic}(X)$ be the class of $\pi^*$ of a line in $\mathbb P^2$. With this set up, let us note down the following crucial theorem.\
	
	\begin{theorem}\label{cubic}
		Let $X$ be a smooth cubic surface in $\mathbb P^3$. Then :
		
		$(a)$ $\text{Pic}(X) \cong \mathbb Z^7$.\
		
		$(b)$ the intersection pairing on $X$ is given by $l^2=1, e_i^2=-1, l.e_i=0, e_i.e_j=0$ for $i \neq j$.\
		
		$(c)$ the hyperplane section $H$ is $3l-\sum_{i=1}^{6} e_i$.\
		
		$(d)$ the canonical class is $K_X=-H=-3l+\sum_{i=1}^{6} e_i$.\
		
		$(e)$ if $D$ is any effective divisor on $X$, $D \sim al-\sum_{i=1}^{6} b_ie_i$, then the degree of $D$, as a curve in $\mathbb P^3$, is $d=3a-\sum_{i=1}^{6} b_i$.\
		
		$(f)$ the self-intersection of $D$ is $D^2=a^2-\sum_{i=1}^{6}b_i^2$.\
	\end{theorem}

    \begin{proof}
    	See \cite{H}, Chapter-$V$, Proposition $4.8$.\
    	\end{proof}
    
    At this point, let us note down the Riemann-Roch theorem for line bundles over a nonsingular cubic hypersurface $X \subset \mathbb P^3$. Let $D$ be divisor on $X$, then the Riemann-Roch theorem for $\mathcal O_X(D)$ is given by
    \begin{align}\label{RR}
    	\chi(\mathcal O_X(D))= \frac{D(D-K_X)}{2}+1
    \end{align}
    where $K_X$ is the canonical class of $X$.\\
    
The following nefness criteria will follow as a corollary of the ampleness criteria \cite[Chapter-$V$, Corollary $4.13(a)$]{H}. Note that this nefness criteria  will be useful for preparing the explicit list of $1$-away and $2$-away ACM divisors on a smooth cubic surface in section $\S \ref{cl12}$.\
\begin{theorem}\label{N}
	Let $D \sim al-\sum_{i=1}^{6} b_ie_i$ be a divisor class on a nonsingular cubic surface $X \subset \mathbb P^3$. Then $D$ is nef  $\iff$ $b_i \ge 0$ for each $i$, and $a \ge b_i+b_j$ for each $i,j$, and $2a \ge \sum_{i\neq j} b_i$ for each $j$
\end{theorem}

\begin{proof}
	By the ampleness criterion in \cite[Chapter-$V$, Corollary $4.13(a)$]{H}, we have $D$ is ample  $\iff$ $b_i > 0$ for each $i$, and $a > b_i+b_j$ for each $i,j$, and $2a > \sum_{i\neq j} b_i$ for each $j$. By Kleiman's criterion, the nef cone $\text{Nef}(X)$ is the closure of the ample cone $\overline{\text{Amp}(X)}$. Thus any integral nef divisor $D$ can be written as $D=\lim_{n \to \infty}D_n$, where $D_n=a_nl-\sum_{i=1}^{6}b_n(i)e_i$ are ample $\mathbb{Q}$-divisors. There exists natural number $N_n$ such that $N_n \cdot D_n=N_na_nl-N_n(\sum_{i=1}^{6}b_n(i)e_i)$ is an integral ample divisor. By \cite[Chapter-$V$, Corollary $4.13(a)$]{H}, ampleness of $N_n \cdot D_n$ is equivalent to $N_n \cdot b_n(i) >0$ for each $i$, and $N_na_n > N_nb_n(i)+N_nb_n(j)$ for each $i,j$, and $2N_na_n > \sum_{i\neq j} N_nb_n(i)$ for each $j$. Since $N_n \ge 1$, we conclude that $b_n(i) >0$ for each $i$, and $a_n > b_n(i)+b_n(j)$ for each $i,j$, and $2a_n > \sum_{i\neq j} b_n(i)$ for each $j$. From $D=\lim_{n \to \infty}D_n$, we get the following
	\begin{equation}
		a=\lim_{n \to \infty}a_n, b_i=\lim_{n \to \infty}b_n(i).
	\end{equation}
	Hence, it follows that $b_i \ge 0$, and $a-b_i-b_j= \lim_{n \to \infty}(a_n-b_n(i)-b_n(j)) \ge 0$, and $2a-\sum_{i\neq j}b_i=\lim_{n \to \infty}(2a_n-\sum_{i\neq j}b_n(i)) \ge 0$.\\
	
	For the converse, we assume that for a divisor $D\sim al-\sum_{i=1}^{6} b_ie_i$, we have $b_i \ge 0$ for each $i$, and $a \ge b_i+b_j$ for each $i,j$, and $2a \ge \sum_{i\neq j} b_i$ for each $j$. Then we consider the $\mathbb{Q}$-divisors $D_n=(a+\frac{1}{n})l-\sum_{i=1}^{6}(b+\frac{1}{5n})e_i$ for all $n \in \mathbb{N}$. Then from the ampleness criterion \cite[Chapter-$V$, Corollary $4.13(a)$]{H}, it follows that $D_n$'s are ample $\mathbb{Q}$-divisors for all $n \in \mathbb{N}$. We also have $D=\lim_{n \to \infty}D_n$. Hence, it follows that $D$ is nef.

	\
\end{proof}




Next, we recall the generalized version of Castelnuovo-Mumford regularity criteria for any smooth projective variety.\

\begin{definition}(Regularity) \normalfont  Let $X$ be a projective variety and $B$ an ample line bundle on $X$ that is generated by its global sections. A coherent sheaf $\mathcal F$ on $X$ is $m$-regular with respect to $B$ if $H^i(X, \mathcal F \otimes B^{\otimes (m-i)})=0$ for $i >0$.\
	
\end{definition}

\begin{theorem}\label{CM}
	Let $\mathcal F$ be an $m$-regular sheaf with respect to $B$. Then for every $k \geq 0$.\
	
	$(i)$ $\mathcal F \otimes B^{\otimes (m+k)}$ is generated by its global sections.\
	
	$(ii)$ The natural maps $H^0(X, \mathcal F \otimes B^{\otimes m}) \otimes H^0(X,  B^{\otimes k}) \to H^0(X, \mathcal F \otimes B^{\otimes (m+k)}) $ are surjective.\
	
	$(iii)$ $\mathcal F$ is $(m+k)$-regular with respect to $B$.
\end{theorem}

\begin{proof}
See \cite{PAG}, Theorem $1.8.5$.\
\end{proof}


Next, we recall the definitions of the different types of bundles concerned. Although all these definitions can be given more generally i.e. for polarized projective varieties, for our purpose (i.e. for application on a smooth cubic surface), we only note down the definitions for smooth embedded projective varieties.\

\begin{definition}\label{In}(Initialized) \normalfont Let $X$ be a smooth projective variety and $H$ be a very ample divisor on $X$ giving the embedding $X \subset \mathbb P^n$. A vector bundle $E$ on $X$ is called initialized with respect to $H$ if $H^0(X, E) \neq 0$ and $H^0(X, E\otimes \mathcal O_X(-H)) =0$.\
\end{definition}

\begin{remark}
\normalfont	Note that by Serre duality and Serre vanishing, any vector bundle on a smooth embedded projective variety $(X,H)$ can be made initialized.
\end{remark}

\begin{definition}(ACM)
\normalfont	Let $X$ be a smooth projective variety and $H$ be a very ample divisor on $X$ giving the embedding $X \subset \mathbb P^n$. A vector bundle $E$ on $X$ is called ACM with respect to $H$ if $H^i(X, E \otimes \mathcal O_X(H)^{\otimes t}) =0$ for all $t \in \mathbb Z$ and $0 < i < \text{dim}(X)$.\
\end{definition}

Now we recall the classification theorem of ACM line bundles on Del-pezzo surfaces by Pons-Llopis and Tonini. This theorem is a crucial ingredient for the characterization of $1$-away and $2$-away ACM line bundles over a cubic surface.\

\begin{theorem}\label{Ton}
	Let $X \subset \mathbb P^d$ be a del Pezzo surface of degree $d$ embedded through the very ample divisor $H=-K_X$ and $\mathcal L$ be a line bundles on $X$. Then the following are equivalent:
	
	$(1)$ $\mathcal L$ is initialized and ACM.\
	
	$(2)$  $\mathcal L$ is initialized and $h^1(\mathcal L(-1)) = h^1(\mathcal L(-2))=0$.\
	
	$(3)$ $\mathcal L \cong \mathcal O_X$ or $\mathcal L \cong \mathcal O_X(D)$ where $D$ is a divisor such that $D^2=D.H-2$ and $0 <D.H \leq H^2$.\
	
	$(4)$ $\mathcal L \cong \mathcal O_X$ or $\mathcal L \cong \mathcal O_X(D)$, where $D$ is a rational normal curve on $X$ with $\text{deg}(D) \leq d$.\
	
\end{theorem}

\begin{proof}
	See \cite{PLT}, Theorem $4.5$.\
\end{proof}

\begin{definition}($\ell$-away ACM)
\normalfont	Let $X$ be a smooth projective variety and $H$ be a very ample divisor on $X$ giving the embedding $X \subset \mathbb P^n$. A vector bundle $E$ on $X$ is called $\ell$-away ACM with respect to $H$ if there are exactly $\ell$ pairs $(i,t) \in \mathbb Z^{2}$ with $0 <i <\text{dim}(X)$ such that $H^i(X, E \otimes \mathcal O_X(H)^{\otimes t}) \neq 0$.\
\end{definition}

\begin{remark}
\normalfont Let $\mathcal L$ be a line bundle on a smooth cubic surface $X \subset \mathbb P^3$. By Serre duality, it easily follows that: if $\mathcal L$ is ACM (respectively $\ell$-away ACM) then so is $\mathcal L^*$.\
\end{remark}

\begin{definition}(Ulrich)
\normalfont	Let $X$ be a smooth projective variety and $H$ be a very ample divisor on $X$ giving the embedding $X \subset \mathbb P^n$. A vector bundle $E$ on $X$ is called Ulrich with respect to $H$ if $E$ is initialized and ACM bundle and $h^0(E)= \text{deg}_H(X). \text{rank}(E)$ or equivalently $H^i(E(-t))=0$ for $i \geq 0$ and $1 \leq t \leq \text{dim}(X)$.\
\end{definition}

\begin{remark}
\normalfont	  Note that [\cite{Eisenbud}, Proposition $2.1$], [\cite{book}, Theorem 3.2.9] suggests that Ulrich bundles can be characterized in many different ways.\     
	
\end{remark}

It is nontrivial to construct/find Ulrich bundles on a smooth projective variety. As an approximation of Ulrich bundles, one has the so called `weakly Ulrich' bundles which are almost Ulrich in the following sense.\

\begin{definition}\label{Wul}(weakly-Ulrich)
\normalfont	Let $X$ be a smooth projective variety and $H$ be a very ample divisor on $X$ giving the embedding $X \subset \mathbb P^n$. An initialized vector bundle $E$ on $X$ is called weakly Ulrich with respect to $H$ if $E$ satisfies the following cohomological properties:\
	
	$(i)$ $H^i(E(-t))=0$ for $1 \leq i \leq \text{dim}(X)$ and $t \leq i-1 \leq \text{dim}(X)-1$, and\
	
	$(ii)$ $H^i(E(-t))=0$ for $0 \leq i \leq \text{dim}(X)-1$ and $t \geq i+2$.\
\end{definition}

\begin{remark}
\normalfont Every Ulrich bundle is weakly Ulrich, but the converse is not true, e.g. $\Omega^1_{\mathbb P^n}(2)$ on $\mathbb P^n (n \geq 2)$. Note that if $\text{dim}(X)=2$, then a weakly Ulrich bundle is either ACM or $1$-away ACM or $2$-away ACM (with nonvanishing twists $-1$ or $-2$ or both).\
	
\end{remark}


	


	
	
		
		
		
		

\section{Technical results}\label{tech}

In this section, we prepare some results on line bundles over a cubic surface which are important ingredients for the proof of the  main theorems in the upcoming sections $\S\ref{cl12}$ and $\S\ref{ex}$. For convenience, we divide the later part of this section into two subsections. In the first subsection i.e. in  $\S\ref{DB}$, we establish a necessary degree bound result for $\ell(\geq 2)$-away ACM line bundles on $X$. In the second subsection i.e. in  $\S\ref{12}$, we motivate the studies carried out in the upcoming sections by documenting examples of $\ell$-away ACM line bundles of low degree on $X$. Throughout this section, $X$ will stand for a nonsingular cubic hypersurface in $\mathbb P^3$ and $H= \mathcal O_{\mathbb P^3}(1)|_X$.\\

The following lemma will be crucial for proving the sufficient directions of the classification theorems (in  section $\S\ref{cl12}$) for $1$-away and $2$-away ACM line bundles on a nonsingular cubic surface $X$. To be more precise, for a non-zero effective divsior $D$ on $X$, this lemma guarantees the vanishing of $h^1(\mathcal O_X(D+tH))$ for all integers $t>k$ (respectively $t<k$), once it is known that $h^1(\mathcal O_X(D+kH))=0$ for some $k \in \mathbb Z$.\

\begin{lemma}\label{L1}
Let $D$ be a non-zero effective divisor on $X$. Then the following holds:

$(i)$ Let $h^0(\mathcal O_X(D+tH)) \leq \frac{D^2+(2t+1)D.H+3t(t+1)+2}{2}$ for some $t \geq -1$.  Then for all $s \geq t$, $h^0(\mathcal O_X(D+sH)) \leq \frac{D^2+(2s+1)D.H+3s(s+1)+2}{2}$ 

$(ii)$ Let $h^0(\mathcal O_X((t-1)H-D)) \leq \frac{D^2-(2t-1)D.H+3t(t-1)+2}{2}$ for some $t > \frac{D.H}{3}$. Then for all $s \geq t$, $h^0(\mathcal O_X((s-1)H-D)) \leq \frac{D^2-(2s-1)D.H+3s(s-1)+2}{2}$ 
	
\end{lemma}

\begin{proof}
For $(i)$, assume $h^0(\mathcal O_X(D+tH)) \leq \frac{D^2+(2t+1)D.H+3t(t+1)+2}{2}$ for some $t \geq -1$. By principle of mathematical induction, it suffices to establish the inequality for $s=t+1$. Taking cohomology to the following S.E.S:
\begin{align}
		0 \to \mathcal O_X(D+tH) \to \mathcal O_X(D+(t+1)H) \to \mathcal O_H(D+(t+1)H) \to 0 
\end{align}
we get $h^0(\mathcal O_X(D+(t+1)H)) \leq h^0(\mathcal O_X(D+tH)) +h^0(\mathcal O_H(D+(t+1)H)) \leq \frac{D^2+(2t+1)D.H+3t(t+1)+2}{2}+ h^0(\mathcal O_H(D+(t+1)H))$. Since $t \geq -1$, by Riemann-Roch theorem for line bundles on curves, we get $h^0(\mathcal O_H(D+(t+1)H))= D.H+3(t+1)$. This gives us, $h^0(\mathcal O_X(D+(t+1)H)) \leq \frac{D^2+(2(t+1)+1)D.H+3(t+1)((t+1)+1)+2}{2}$.\\

For $(ii)$, assume $h^0(\mathcal O_X((t-1)H-D)) \leq \frac{D^2-(2t-1)D.H+3t(t-1)+2}{2}$ for some $t > \frac{D.H}{3} $. By principle of mathematical induction, it suffices to establish the inequality for $s=t+1$. Taking cohomology to the following S.E.S:
\begin{align}
	0 \to \mathcal O_X((t-1)H-D) \to \mathcal O_X(tH-D) \to \mathcal O_H(tH-D) \to 0 
\end{align}
we get $h^0(\mathcal O_X(tH-D) \leq h^0(\mathcal O_X((t-1)H-D)) +h^0(\mathcal O_H(tH-D)) \leq \frac{D^2-(2t-1)D.H+3t(t-1)+2}{2} + h^0(\mathcal O_H(tH-D))$. Since $t > \frac{D.H}{3}$, by Riemann-Roch theorem for line bundles on curves, we get $h^0(\mathcal O_H(tH-D)) = 3t-D.H$. This gives us $h^0(\mathcal O_X(tH-D) \leq \frac{D^2-(2(t+1)-1)D.H+3(t+1)t+2}{2}$.

\end{proof}

Now we prepare a technical result which is crucial for the proofs of the two main theorems of this article, i.e. Theorem \ref{T31} and Theorem \ref{T35} regarding the classification of $1$-away and $2$-away ACM line bundles. Being technical, we isolate a few cases i.e. divisors of degrees $2,3,4$ with specific self intersection numbers. To be more precise, for certain divisors $D$ on $X$, the following proposition guarantess the vanishing of $h^1(\mathcal O_X(D+kH))=0$ for some $k \in \mathbb Z$ as the base case of the induction process which then allows us to further apply the previous lemma (i.e. Lemma \ref{L1}) to complete the induction process.\



	


\begin{proposition}\label{P2}
	Let $\Gamma$ be a non-zero effective divisor on $X$. Then the following holds:
	
	$(i)$ Let $(\Gamma^2, \Gamma.H)=(-2,2)$. Then $h^0(\mathcal O_X(\Gamma))=1$.\
	
	$(ii)$ Let $(\Gamma^2, \Gamma.H)=(-3,3)$. Then $h^0(\mathcal O_X(\Gamma))=1$.\
	
	$(iii)$ Let $(\Gamma^2, \Gamma.H)=(-1,3)$. Then $h^0(\mathcal O_X(\Gamma))=2$.\
	
	$(iv)$ Let $(\Gamma^2, \Gamma.H)=(0,4)$. Then $h^0(\mathcal O_X(\Gamma))=3$.\

\end{proposition}

\begin{proof}
Let $\Gamma$ be as in any one of the above four cases. Then its arithmetic genus $p_a(\Gamma) <0$. Thus $\Gamma$ is either not reduced or not irreducible. Therefore, there is a nontrivial effective decomposition $\Gamma=\Gamma_1+\Gamma_2$. Without loss of generality, we can assume that if $\Gamma$ is not irreducible, then $\Gamma_1$ and $ \Gamma_2$ can't have a common component. With this convention of the decomposition, we proceed to analyse each of the four cases as follows.\
	
$(i)$ Let $(\Gamma^2, \Gamma.H)=(-2,2)$. Since $\Gamma.H=2$, we must have $\Gamma_i.H=1$ ($i=1,2$) i.e. $\Gamma_i$'s are lines on $X$ and hence $\Gamma^2_i =-1$ ($i=1,2$). By considering the expression of $\Gamma^2$, we get $\Gamma_1.\Gamma_2=0$. Therefore, we have $\Gamma = \Gamma_1+\Gamma_2$ with $\Gamma_i.H=1$ ($i=1,2$) and $\Gamma_1.\Gamma_2=0$.  Tensorizing  the S.E.S:
\begin{align}\label{ss}
	0 \to \mathcal O_X(-\Gamma_1) \to \mathcal O_X \to \mathcal O_{\Gamma_1} \to 0
\end{align}
by $\mathcal O_X(\Gamma_1+\Gamma_2)$ and taking cohomology, we get $h^0(\mathcal O_X(\Gamma))=h^0(\mathcal O_X(\Gamma_1+\Gamma_2))=h^0(\mathcal O_X(\Gamma_2))$ (since $\Gamma_1$ is a line on $X$ and hence ACM, $\Gamma_1^2=-1$ and $\Gamma_1.\Gamma_2=0$). Again tensorizing the S.E.S \ref{ss} (with $\Gamma_1$ replaced by $\Gamma_2$) by $\mathcal O_X(\Gamma_2)$ and taking cohomology,
we get $h^0(\mathcal O_X(\Gamma_2))=h^0(\mathcal O_X)=1$ (since $\Gamma_2^2=-1$).\\

$(ii)$ Let  $(\Gamma^2, \Gamma.H)=(-3,3)$.  Without loss of generality, assume that $\Gamma_1.H=1, \Gamma_2.H=2$. Note that $\Gamma_2 \neq 2.\Gamma_1$, otherwise $\Gamma^2=-9$ (a contradiction). By the choice of decomposition, we must have, $\Gamma_1.\Gamma_2 \geq 0$ and hence considering the expression of $\Gamma^2$, we see that $\Gamma_2^2 \leq -2$ and hence $p_a(\Gamma_2) <0$. Thus, there is a nontrivial effective decomposition $\Gamma_2=\Gamma_3+\Gamma_4$ with $\Gamma_i.H=1$, $i\in \{3,4\}$. Note that $\Gamma_3 \neq \Gamma_4$, otherwise considering the expression of $\Gamma^2$, one obtains $4\Gamma_1.\Gamma_3=2$, a contradiction. Therefore, $\Gamma_2^2=-2, \Gamma_3.\Gamma_4=0$. Then by the choice of the decomposition and considering the expression for $\Gamma^2$, it can be deduced that $\Gamma=\Gamma_1+\Gamma_3+\Gamma_4$ with $\Gamma_i.\Gamma_j=0$ for $i \neq j \in \{1,3,4\}$. Tensorizing the S.E.S \ref{ss} by $\mathcal O_X(\Gamma_1+\Gamma_3+\Gamma_4)$ and taking cohomology, we get $h^0(\mathcal O_X(\Gamma))= h^0(\mathcal O_X(\Gamma_3+\Gamma_4))$ (since $\Gamma_1^2=-1$, $\Gamma_i.\Gamma_j=0$ for $i \neq j \in \{1,3,4\}$). Tensorizing the S.E.S \ref{ss} (with $\Gamma_1$ replaced by $\Gamma_3$) by $\mathcal O_X(\Gamma_3+\Gamma_4)$ and taking cohomology, we get $ h^0(\mathcal O_X(\Gamma_3+\Gamma_4))= h^0(\mathcal O_X(\Gamma_4))$ (since $\Gamma_3^2=-1$, $\Gamma_i.\Gamma_j=0$ for $i \neq j \in \{1,3,4\}$). Finally, tensorizing the S.E.S \ref{ss} (with $\Gamma_1$ replaced by $\Gamma_4$) by $\mathcal O_X(\Gamma_4)$ and taking cohomology, we get $ h^0(\mathcal O_X(\Gamma_4))= h^0(\mathcal O_X)=1$ (since $\Gamma_4^2=-1$).\\


$(iii)$ Let  $(\Gamma^2, \Gamma.H)=(-1,3)$.  Without loss of generality, assume that $\Gamma_1.H=1, \Gamma_2.H=2$. Note that $\Gamma_2 \neq 2.\Gamma_1$, otherwise $\Gamma^2=-9$ (a contradiction). By the choice of decomposition, we must have, $\Gamma_1.\Gamma_2 \geq 0$ and hence considering the expression of $\Gamma^2$, we see that $\Gamma_2^2 \leq 0$. We first consider the situation where  $\Gamma_2^2 \leq -1$. Since $p_a(\Gamma_2) <0$, there is a nontrivial effective decomposition $\Gamma_2=\Gamma_3+\Gamma_4$ with $\Gamma_i.H=1$, $i\in \{3,4\}$. If $\Gamma_3=\Gamma_4$, then it can be deduced that $\Gamma=\Gamma_1+2.\Gamma_4$ with $\Gamma_1.\Gamma_4=1$ (by the choice of decomposition and considering the expression of $\Gamma^2$). Tensorizing the S.E.S \ref{ss} (with $\Gamma_1$ replaced by $\Gamma_4$)
by $\mathcal O_X(\Gamma_1+2.\Gamma_4)$ and taking cohomology we have $h^0(\mathcal O_X(\Gamma))=h^0(\mathcal O_X(\Gamma_1+\Gamma_4))$ (since $\Gamma_4^2=-1, \Gamma_1.\Gamma_4=1$). Tensorizing the S.E.S \ref{ss} by $\mathcal O_X(\Gamma_1+\Gamma_4)$ and noting that $\mathcal O_X(\Gamma_1)$ is ACM (by Theorem \ref{Ton}, $(3)$), we get $h^0(\mathcal O_X(\Gamma_1+\Gamma_4))=2$.  If $\Gamma_3 \neq \Gamma_4$, then without loss of generality, one can rewrite $\Gamma=\Gamma_1+\Gamma_3+\Gamma_4$ with $\Gamma_3.\Gamma_4=0, \Gamma_1.\Gamma_4=0, \Gamma_1.\Gamma_3=1$ (by the choice of decomposition and considering the expressions of $\Gamma_2^2, \Gamma^2$). Tensorizing the S.E.S \ref{ss} (with $\Gamma_1$ replaced by $\Gamma_4$) by $\mathcal O_X(\Gamma)$  and taking cohomology, we have $h^0(\mathcal O_X(\Gamma))=h^0(\mathcal O_X(\Gamma_1+\Gamma_3))$ (since $\Gamma_4^2=-1, \Gamma_1.\Gamma_4=0$, $\Gamma_3.\Gamma_4=0$). Again tensorizing the S.E.S \ref{ss} (with $\Gamma_1$ replaced by $\Gamma_3$) by $\mathcal O_X(\Gamma_1+\Gamma_3)$  and taking cohomology, we have $h^0(\mathcal O_X(\Gamma_1+\Gamma_3))=2$ (since $\Gamma_3^2=-1, \Gamma_1.\Gamma_3=1$ and $\mathcal O_X(\Gamma_1)$ is ACM).\


 Next, we deal with the remaining case when $\Gamma_2^2=0$. Tensorizing the S.E.S \ref{ss} by $\mathcal O_X(\Gamma_1+\Gamma_2)$ and taking cohomology, we yield $h^0(\mathcal O_X(\Gamma)) = h^0(\mathcal O_X(\Gamma_2))$ (since $\Gamma_1^2=-1, \Gamma_1.\Gamma_2=0$). Since by  Theorem \ref{Ton}, $(3)$, $\mathcal O_X(\Gamma_2)$ is ACM, by \ref{RR} applied to $\mathcal O_X(\Gamma_2)$, we obtain  $h^0(\mathcal O_X(\Gamma_2))=2$ and we are done.\\

 $(iv)$ Let  $(\Gamma^2, \Gamma.H)=(0,4)$. As before, we have $p_a(\Gamma)<0$ and thus we have a nontrivial effective decomposition of $\Gamma$. We maintain the convention that if $\Gamma$ is not irreducible, then the decomposed parts do not have a common component. We divide the analysis of this case into two parts:
 
 \underline{\textbf{Part 1: $\Gamma$ contains a line say $D_1$ as a component}}
 
 We arrange the decomposition of $\Gamma$ taking into account the maximum possible multiplicity of $D_1$. Since the decomposed parts do not have a common component, it is enough to consider the following seven cases.
 \begin{itemize}
 	\item $\Gamma=4D_1$. Since $D_1$ is a line on $X$ and hence $D_1^2=-1$, one obtains $\Gamma^2=-16$, a contradiction. Therefore, this case can't occur.\
 	
 	\item $\Gamma=3D_1+D_2$, where $D_2$ is a line on $X$ with $D_1.D_2 \geq 0$. Since $D_i^2=-1$ for $i=1,2$, considering the expression of $\Gamma^2$, one obtains $6D_1.D_2=10$, a contradiction. Therefore, this case can't occur.\
 	
 	\item $\Gamma=2D_1+2D_2$, where $D_2$ is a line on $X$ with $D_1.D_2 \geq 0$. Since $D_i^2=-1$ for $i=1,2$, considering the expression of $\Gamma^2$, one obtains $D_1.D_2=1$. Note that by \ref{RR} applied to $\mathcal O_X(\Gamma)$, one gets $h^0(\mathcal O_X(\Gamma)) \geq 3$. Therefore, it is enough to show $h^0(\mathcal O_X(\Gamma)) \leq 3$. Tensorizing \ref{ss} (with $\Gamma_1$ replaced by $D_1$) by $\mathcal O_X(\Gamma)$ and taking cohomology, we get $h^0(\mathcal O_X(\Gamma)) \leq 1+h^0(\mathcal O_X(D_1+2D_2))$. Since $D_1+D_2$ is ACM by Theorem \ref{Ton}, $(3)$, tensorizing the S.E.S \ref{ss} (with $\Gamma_1$ replaced by $D_2$) by $\mathcal O_X(D_1+2D_2)$  and taking cohomology, we get $h^0(\mathcal O_X(D_1+2D_2))=h^0(\mathcal O_X(D_1+D_2))$. The fact that $h^0(\mathcal O_X(D_1+D_2))=2$ follows from applying \ref{RR} to $\mathcal O_X(D_1+D_2)$ and we are through.\

 	
 	\item $\Gamma=2D_1+D_2+D_3$, where $D_2,D_3$ are lines on $X$ with $D_1.D_2 \geq 0, D_1.D_3 \geq 0, D_2.D_3 \geq 0$. Since $D_i^2=-1$ for $i=1,2$, considering the expression of $\Gamma^2$, one obtains without loss of generality $D_1.D_2=1, D_1.D_3=0, D_2.D_3=1$. Tensorizing the S.E.S \ref{ss} (with $\Gamma_1$ replaced by $D_1$) by $\mathcal O_X(\Gamma)$ and taking cohomology, we get $h^0(\mathcal O_X(\Gamma)=h^0(\mathcal O_X(D_1+D_2+D_3))$. Since $D_1+D_2$ is ACM by Theorem \ref{Ton},$(3)$, tensorizing the S.E.S \ref{ss} (with $\Gamma_1$ replaced by $D_3$) with $\mathcal O_X(D_1+D_2+D_3)$ and taking cohomology, we get $h^0(\mathcal O_X(D_1+D_2+D_3))=h^0(\mathcal O_X(D_1+D_2))+1$. The fact that $h^0(\mathcal O_X(D_1+D_2))=2$ follows from applying \ref{RR} to $\mathcal O_X(D_1+D_2)$ and we are through.\

 	
 	\item $\Gamma=2D_1+Q$, where $Q$ is a reduced and irreducible divisor on $X$ with $H.Q=2, D_1.Q \geq 0$. It is easy to see that $Q^2=0$ and hence considering the expression of $\Gamma^2$, one obtains $D_1.Q =1$. Again as before, one can use S.E.S \ref{ss} to get $h^0(\mathcal O_X(\Gamma))= h^0(\mathcal O_X(D_1+Q))$. Since $\mathcal O_X(Q)$ is ACM by Theorem \ref{Ton}, $(3)$, tensorizing the S.E.S \ref{ss} (with $\Gamma_1$ replaced by $D_1$) by $\mathcal O_X(D_1+Q)$ and taking cohomology we get $h^0(\mathcal O_X(D_1+Q))=1+h^0(\mathcal O_X(Q))$. The fact that $h^0(\mathcal O_X(Q))=2$ follows from applying \ref{RR} to $\mathcal O_X(Q)$ and we are through.\

 	\item $\Gamma=D_1+D_2+D_3+D_4$, where $D_i$'s are lines on $X$ for $i=1,2,3,4$ with $D_i.D_j \geq 0$ for $i \neq j \in \{1,2,3,4\}$. Since $D_i^2=-1$ for $i \in \{1,2,3,4\}$, considering the expression of $\Gamma^2$, without loss of generality, we can write $D_1.D_2=D_1.D_3=1$ and the remaining such pairwise intersection numbers are $0$. Tensorizing the S.E.S \ref{ss} (with $\Gamma_1$ replaced by $D_4$) by $\mathcal O_X(\Gamma)$ and taking cohomology, we get $h^0(\mathcal O_X(\Gamma)=h^0(\mathcal O_X(D_1+D_2+D_3))$. Since $D_1+D_2$ is ACM by Theorem \ref{Ton},$(3)$, tensorizing the S.E.S \ref{ss} (with $\Gamma_1$ replaced by $D_3$) by $\mathcal O_X(D_1+D_2+D_3)$ and taking cohomology, we get $h^0(\mathcal O_X(D_1+D_2+D_3))=h^0(\mathcal O_X(D_1+D_2))+1$. The fact that $h^0(\mathcal O_X(D_1+D_2))=2$ follows from applying \ref{RR} to $\mathcal O_X(D_1+D_2)$ and we are through.\

    \item $\Gamma=D_1+C$, where $C$ is a reduced and irreducible divisor on $X$ with $C.H=3, D_1.C \geq 0$. It is easy to see that $C^2=1$ and hence considering the expression of $\Gamma^2$, one obtains $D_1.C =0$. Tensorizing the S.E.S \ref{ss} (with $\Gamma_1$ replaced by $D_1$) with $\mathcal O_X(\Gamma)$ and taking cohomology, we get $h^0(\mathcal O_X(\Gamma))=h^0(\mathcal O_X(C))$. The fact that $h^0(\mathcal O_X(C))=3$ follows from applying \ref{RR} to $\mathcal O_X(C)$ and we are through.
 	\end{itemize}
\underline{\textbf{Part 2: $\Gamma$ does not contain a line  as a component}}

In this case, it is enough to consider the following two possibilities:
\begin{itemize}

 \item $\Gamma= \Gamma_1+\Gamma_2$, where $\Gamma_1$ and $\Gamma_2$ are irreducible and reduced divisor on $X$ with $\Gamma_i.H=2$ for $i=1,2$ and $\Gamma_1.\Gamma_2 \geq 0$. It is easy to see that $\Gamma_i^2=0$ for $i=1,2$ and hence considering the expression of $\Gamma^2$, one gets $\Gamma_1. \Gamma_2=0$. Since $\Gamma_i$'s are ACM by Theorem \ref{Ton}, $(3)$, tensorizing the S.E.S \ref{ss} with $\mathcal O_X(\Gamma)$ and taking cohomology, we get $h^0(\mathcal O_X(\Gamma))=1+h^0(\mathcal O_X(\Gamma_2))$. The fact that $h^0(\mathcal O_X(\Gamma_2))=2$ follows from applying \ref{RR} to $\mathcal O_X(\Gamma_2)$ and we are done.\
 
 \item $\Gamma= 2.\Gamma_1$, where $\Gamma_1$ is an irreducible and reduced divisor on $X$ with $ \Gamma_1.H=2$. Since $\Gamma_1$ is ACM by Theorem \ref{Ton}, $(3)$, tensorizing the S.E.S \ref{ss} with $\mathcal O_X(\Gamma)$ and taking cohomology, we get $h^0(\mathcal O_X(\Gamma))=1+h^0(\mathcal O_X(\Gamma_1))$. The fact that $h^0(\mathcal O_X(\Gamma_1))=2$ follows from applying \ref{RR} to $\mathcal O_X(\Gamma_1)$ and we are done.\
  
\end{itemize}

\end{proof}

Let $X \subset \mathbb{P}^3$ be a cubic surface and $H$ be an elliptic curve in the linear system $|\mathcal{O}_{\mathbb{P}^3}(1)|_{X}|$. We will need the following lemma to prove the connectedness of the first cohomology module in Theorem \eqref{Ext}. This lemma can also be of general importance.\ 
\begin{lemma}\label{general}
	Let $\mathcal{L}$ be a line bundle on $X$ such that $H^1(X, \mathcal{L})=0$. Then for all $n>0$ we will have
	\begin{itemize}
		\item[(1)] The vanishing $H^1(X, \mathcal{L}-nH)=0$ if $\text{deg}(\mathcal{L}_{|H}) < 0$.
		\item[(2)] The vanishing $H^1(X, \mathcal{L}+nH)=0$ if $\text{deg}(\mathcal{L}_{|H}) \ge 0$.
	\end{itemize}
\end{lemma}

\begin{proof}
	For any $m \in \mathbb{Z}$, we will have the following S.E.S
	\begin{equation}
		\label{lawaylemma}
		0 \rightarrow \mathcal{O}_X(\mathcal{L}+(m-1)H) \rightarrow \mathcal{O}_X(\mathcal{L}+mH) \rightarrow \mathcal{O}_H(\mathcal{L}+mH) \rightarrow 0
	\end{equation}
	The degree of the line bundle $\mathcal{O}_H(\mathcal{L}+mH)$ over the elliptic curve $H$ is $\text{deg}(\mathcal{L}_{|H}) + m \cdot H^2=\text{deg}(\mathcal{L}_{|H}) + 3m$. Let us assume $\text{deg}(\mathcal{L}_{|H}) < 0$ and $m \le 0$. Thus, $\text{deg}(\mathcal{L}_{|H}) + 3m <0$ and hence $H^0(H, \mathcal{O}_H(\mathcal{L}+mH))=0$. We set $n=-m$. Then taking the long exact sequence for \eqref{lawaylemma}, we get the exact sequence
	$$ H^0(H, \mathcal{O}_H(\mathcal{L}-nH)) \rightarrow H^1(X, \mathcal{O}_X(\mathcal{L}+(-n-1)H)) \rightarrow H^1(X, \mathcal{O}_X(\mathcal{L}-nH)) \rightarrow H^1(H, \mathcal{O}_H(\mathcal{L}-nH)) $$ 
	Since $H^0(H, \mathcal{O}_H(\mathcal{L}-nH))=0$ for all $n \ge 0$, we get $H^1(X, \mathcal{O}_X(\mathcal{L}+(-n-1)H)) \subseteq H^1(X, \mathcal{O}_X(\mathcal{L}-nH))$. Since $H^1(X, \mathcal{L})=0$, by induction we get $H^1(X, \mathcal{L}-nH)=0$ for all $n >0$.\\
	
	For the other case, we have the assumption $\text{deg}(\mathcal{L}_{|H}) \ge 0$. Thus, for any integer $m > 0$, we have $\text{deg}(\mathcal{L}_{|H}) + 3m >0$. Therefore, $H^1(H, \mathcal{O}_H(\mathcal{L}+mH))=0$. By taking the long exact sequence for the S.E.S \eqref{lawaylemma} we get the exact sequence
	$$H^1(X, \mathcal{O}_X(\mathcal{L}+(m-1)H)) \rightarrow H^1(X, \mathcal{O}_X(\mathcal{L}+mH)) \rightarrow H^1(H, \mathcal{O}_H(\mathcal{L}+mH))$$
	Since $H^1(H, \mathcal{O}_H(\mathcal{L}+mH))=0$ for all $m > 0$, we get the surjection $$H^1(X, \mathcal{O}_X(\mathcal{L}+(m-1)H)) \rightarrow H^1(X, \mathcal{O}_X(\mathcal{L}+mH)).$$
	Since $H^1(X, \mathcal{L})=0$, by induction we get $H^1(X, \mathcal{L}+nH)=0$ for all $n>0$.
\end{proof}

The following lemma characterizes initialized divisors on $X$ satisfying certain precise relation between its degree and self intersection number in terms of nefness. This lemma will be useful in studying $\ell (\geq 2)$-away ACM line bundles on $X$ and preparing the explicit list of divisors in the later part of this paper.\ 

\begin{lemma}\label{nef}
Let $D$ be a non-zero effective divisor on $X$ such that $D^2=D.H-2$   and $D.H \geq 3$. Then $\mathcal O_X(D)$ is initialized  if and only if $D$ is nef.
\end{lemma}

\begin{proof}
Note that the conditions $D^2=D.H-2$ and $D.H \geq 3$ together will give us $h^1(\mathcal O_X(D-H))=0$ and $h^2(\mathcal O_X(D-2H))=0$. This means $\mathcal O_X(D)$ is $0$-regular and therefore, by Castelnuovo-Mumford regularity i.e. Theorem \ref{CM}, $(i)$, globally generated and hence nef.\

Conversely, assume that $D$ is nef. Since $D^2>0$ and hence $D$ is big, Kawamata-Viehweg vanishing forces $h^1(\mathcal O_X(D-H)) =0$, whence by \ref{RR},  we obtain $h^0(\mathcal O_X(D-H)) =0$.
	
\end{proof}

\subsection{\textbf{Bounding the degree of the divisors for $\ell \geq 2$ case}}\label{DB}\

In this subsection, we  record a necessary degree bound result for an initialized $\ell (\geq 2)$-away ACM line bundle on $X$ which is crucial for establishing the classification theorem (i.e. Theorem \ref{T35}) of $2$-away ACM line bundles on $X$. We also show that for $\ell \geq 3, (>3)$ respectively this degree bound result can be improved. \

Before proceeding for the degree bound result, we first  prove the following proposition, which will help us to improve the degree bound result obtained in $\ell \geq 2$ case to $\ell \geq 3$ case.\

\begin{proposition}\label{3l}
	Let $\ell \geq 2$. Let $D$ be a non-zero effective divisor on $X$ such that $D$ satisfies the following conditions:\
	
	$(i)$ $(D^2, D.H) =(3\ell-2, 3\ell)$,\
	
	$(ii)$ $D$ is nef,\
	
	$(iii)$ $|(2\ell-1)H-D| = \emptyset$.\
	
	Then $\mathcal O_X(D)$ is $(2\ell-2)$-away ACM
	
\end{proposition}

\begin{proof}
	Since $D$ is nef and big, Kawamata-Viehweg vanishing  forces $h^1(\mathcal O_X(D-H)) =0$. Therefore, $\mathcal O_X(D)$ is $0$-regular and hence by Castelnuovo-Mumford regularity i.e. Theorem \ref{CM},$(iii)$, $h^1(\mathcal O_X(D+tH))=0$ for all $t \geq -1$. Next, let $2 \leq t \leq 2\ell-1$. Since Lemma \ref{nef} forces $\mathcal O_X(D)$ to be initialized, \ref{RR} gives us  $h^2(\mathcal O_X(D-tH))-h^1(\mathcal O_X(D-tH))=\frac{6l-6\ell t+3t^2-3t}{2}<0$.  This means that $h^1(\mathcal O_X(D-tH)) \neq 0$ for all $2 \leq t \leq 2\ell-1$. From assumption $(iii)$ and \ref{RR}, we have $h^1(\mathcal O_X(D-2\ell H))=0$. This coupled with Lemma \ref{L1}, $(ii)$ gives us $h^1(\mathcal O_X(D-tH))=0$ for all $t \geq 2\ell$, whence it follows that $\mathcal O_X(D)$ is $(2\ell-2)$-away ACM.\
\end{proof}


With Proposition \ref{3l} at hand, we are now ready to prove the degree bound result.\

\begin{proposition}\label{degreebound}
	Let $D$ be a non-zero effective divisor on $X$ such that $\mathcal O_X(D)$ is initialized and $\ell$-away ACM with repect to  $H:= \mathcal O_{\mathbb P^3}(1)|_X$. Then the following holds.\
	
	$(i)$ If $\ell \geq 2$, then $D.H \leq 3\ell$.\
	
	$(ii)$ If $\ell \geq 3$, then $D.H \neq 3\ell$ i.e. $D.H \leq 3\ell-1$.\
	
	$(iii)$ If $\ell > \frac{2k}{3}+3$ for some positive integer $k$, then  $D.H \neq 3\ell-k$.\
	
	
\end{proposition}

\begin{proof}
	Let $D$ be a non-zero effective divisor on $X$ such that $\mathcal O_X(D)$ is initialized and $\ell$-away ACM with repect to  $H$.\
	
	$(i)$ On the contrary, assume that $D.H > 3\ell$. Let $1 \leq t \leq \ell$. We first show that if  $H^1(\mathcal O_X(D-tH)) \neq 0$, then $H^1(\mathcal O_X(D-(t+1)H)) \neq 0$. If  $H^1(\mathcal O_X(D-tH)) \neq 0$, then by Serre duality $H^1(\mathcal O_X((t-1)H-D)) \neq 0$. Then applying cohomology to the following S.E.S:
	\begin{align}
		0 \to \mathcal O_X((t-1)H-D) \to \mathcal O_X(tH-D) \to \mathcal O_H(tH-D) \to 0
	\end{align}
	we get $H^1(\mathcal O_X(tH-D)) \neq 0$ i.e. $H^1(\mathcal O_X(D-(t+1)H)) \neq 0$ (by Serre duality). Since $\mathcal O_X(D)$ is $\ell$-away ACM, Theorem \ref{Ton}, $(2)$, forces $h^1(\mathcal O_X(D-H)) = 0$ and $h^1(\mathcal O_X(D-2H)) \neq 0$. To be more precise, in this situation, we must have $H^1(\mathcal O_X(D-tH)) \neq 0$ for all $2 \leq t \leq \ell+1$ (and $H^1$ vanishes otherwise). Applying \ref{RR} to $\mathcal O_X(D-H)$, we get $D^2=D.H-2$. Then application of \ref{RR} to $\mathcal O_X(D-(\ell+2)H)$ gives us $3\ell^2+9\ell+6 \geq (2\ell+2)D.H\geq 6\ell^2+8\ell+2 \implies 3\ell^2-\ell-4 \leq 0$, a contradiction as $\ell \geq 2$. This means we must have $D.H \leq 3\ell$.\\
	
	$(ii)$ Let $\ell \geq 3, D.H=3\ell$. From the first part of the previous paragraph, letting $2 \leq t \leq \ell-1$, we get if  $H^1(\mathcal O_X(D-tH)) \neq 0$, then $H^1(\mathcal O_X(D-(t+1)H)) \neq 0$. Note that by taking cohomology to the following S.E.S:
	\begin{align}
		0 \to \mathcal O_X(D-(\ell+1)H) \to \mathcal O_X(D-\ell H) \to \mathcal O_H(D-\ell H) \to 0
	\end{align}
	we get $H^1(\mathcal O_X(D-(\ell+1)H)) =H^1(\mathcal O_X(D-\ell H)) \neq 0$. Also it can be deduced using  similar arguments as in the first part of case $(i)$ that if $h^1(\mathcal O_X(D-H)) \neq 0$, then $h^1(\mathcal O_X(D-2H)) \neq 0$. Since $\mathcal O_X(D)$ is $\ell$-away ACM, Theorem \ref{Ton}, $(2)$,  forces $H^1(\mathcal O_X(D-tH)) \neq 0$ for all $2 \leq t \leq \ell+1$ (and $H^1$ vanishes otherwise). Applying \ref{RR} to $\mathcal O_X(D-H)$, we get $D^2=3\ell-2$. Note that by Lemma \ref{nef}, $D$ is nef and applying \ref{RR} on $\mathcal O_X((2\ell-1)H-D)$, we get $h^0(\mathcal O_X((2\ell-1)H-D))=0$. Then by Proposition \ref{3l}, $\mathcal O_X(D)$ is $(2\ell-2)$-away ACM, a contradiction as $\ell \neq 2$. Therefore, when $\ell \geq 3$, we must have $D.H \leq 3\ell-1$.\\
	
	$(iii)$ On the contrary, assume that $D.H=3\ell-k$. Since $3\ell-k >3$, taking cohomology to the following S.E.S:
	\begin{align}
		0 \to \mathcal O_X(D-2H) \to \mathcal O_X(D-H) \to \mathcal O_H(D-H) \to 0
	\end{align}
	we have $h^1(\mathcal O_X(D-2H)) \neq 0$. If $2 \leq t <\frac{3\ell-k}{3}$, then applying cohomology to the following S.E.S:
	\begin{align}
		0 \to \mathcal O_X(D-(t+1)H) \to \mathcal O_X(D-tH) \to \mathcal O_H(D-tH) \to 0
	\end{align}
	we obtain that if $h^1(\mathcal O_X(D-tH)) \neq 0$, then $h^1(\mathcal O_X(D-(t+1)H)) \neq 0$. This means that $h^1(\mathcal O_X(D-tH)) \neq 0$ for all $2 \leq t \leq \ell-\frac{k}{3}$. If we can show that for all $\ell-\frac{k}{3} < t <2\ell-\frac{2k}{3}-1 $, $h^1(\mathcal O_X(D-tH)) \neq 0$, then it means that the intermediate cohomology of $\mathcal O_X(D)$ fails to vanish at atleast $2\ell-\frac{2k}{3}-3$ twists, a contradiction (as $2\ell-\frac{2k}{3}-3 >\ell $). Therefore, our work boils down to show that $h^1(\mathcal O_X(D-tH)) \neq 0$ for all $\ell-\frac{k}{3} < t <2\ell-\frac{2k}{3}-1 $. If $h^1(\mathcal O_X(D-H))=0$, then we have $D^2=3\ell-k-2$. By \ref{RR}, we have $h^0(\mathcal O_X(t-1)H-D))-h^1(\mathcal O_X(D-tH))= \frac{6\ell-2k-6t\ell+2tk+3t^2-3t}{2}$. Since $t > \ell-\frac{k}{3}$, we get $\frac{6\ell-2k-6t\ell+2tk+3t^2-3t}{2} < \frac{t(3t+2k-6\ell+3)}{2}$ which is negative as $t <2\ell-\frac{2k}{3}-1 $. Now we are done by noting that, if $h^1(\mathcal O_X(D-H)) \neq 0$, then we have $D^2 \leq 3\ell-k-4$ and hence $h^0(\mathcal O_X(t-1)H-D))-h^1(\mathcal O_X(D-tH)) \leq \frac{6\ell-2k-6t\ell+2tk+3t^2-3t}{2}-1 $, which we have already shown to be negative if $\ell-\frac{k}{3} < t <2\ell-\frac{2k}{3}-1 $. \\
\end{proof}

\subsection{Motivating examples}\label{12}\
	
	To motivate the studies carried out in the upcoming sections, let us end this section by presenting a study on $\ell$-away ACM line bundles of low degree on $X$. To be more precise, we determine the values of $\ell$ for which a non-zero effective divisor $D$ on $X$ with $D.H=1,2$ becomes $\ell$-away ACM.\\
	
	$(1)$ Let $D$ be a non-zero effective divisor on $X$ with $D.H=1$. Since $D$ is a line on $X$ and hence $D^2=-1$, by the Theorem \ref{Ton}, $(3)$, $\mathcal O_X(D)$ is ACM.\\
	
	$(2)$ Let $D$ be a non-zero effective divisor on $X$ with $D.H=2$. Then there are four possibilities:
	\begin{itemize}
		\item $D$ is a reduced and irreducible divisor on $X$ with $D.H=2$. It is easy to see that $D^2=0$. Therefore, by the Theorem \ref{Ton}, $(3)$, $\mathcal O_X(D)$ is ACM.\
		
		\item $D=D_1+D_2$, where $D_i$'s are lines on $X$ for $i=1,2$ with $D_1.D_2=1$. Since $D_i^2=-1$ for $i=1,2$, we have $D^2=0$. Therefore, by the Theorem \ref{Ton}, $(3)$, $\mathcal O_X(D)$ is ACM.\
		
		\item $D=D_1+D_2$, where $D_i$'s are lines on $X$ for $i=1,2$ with $D_1.D_2=0$. In this case, we have, $(D^2, D.H)=(-2,2)$. Since $(D-H).H <0$, $\mathcal O_X(D)$ is initialized. Applying  \ref{RR} to $\mathcal O_X(D-H)$, we see that $h^1(\mathcal O_X(D-H)) = 1$. We will show that $h^1(\mathcal O_X(D+tH)) =0$ for all $t \in \mathbb Z \setminus \{-1\}$.  We divide this into two separate cases and use induction to establish both the cases:\
		
		First, we  show that $H^1(\mathcal O_X(D+tH)) = 0$ for $t\geq 0$. Note that by \ref{RR}, it is enough to show $h^0(\mathcal O_X(D+tH))\leq  \frac{3t^2+7t+2}{2} $ for $t\geq 0$. By Lemma \ref{L1}, $(i)$ if we can show this for $t=0$, then by induction we are done. For $t=0$, this is indeed true by Proposition \ref{P2}, $(i)$ or alternatively by the following argument : observe that $h^1(\mathcal O_X(D)) =h^1(\mathcal O_X(-D-H))=0$ and thus by taking cohomology to the  S.E.S: $0 \to \mathcal O_X(-D-H) \to \mathcal O_X(-H) \to \mathcal O_D(-H) \to 0$ and using Kodaira vanishing.\
		
		Next, we show that $H^1(\mathcal O_X(D-tH)) = 0$ for $t \geq 2$. By \ref{RR}, it is enough to show $h^0(\mathcal O_X((t-1)H-D))\leq  \frac{3t^2-7t+2}{2} $ for $t \geq 2$. By  Lemma \ref{L1}, $(ii)$, if we can show this for $t=2$, then by induction we are through. Note that for $t=2$, the above is true because otherwise, if we assume $|H-D| \neq \emptyset$, then there exists a line $\Gamma \in |H-D|$ and hence  $\Gamma^2=-1$, a contradiction as $\Gamma^2=(H-D)^2=-3$. This shows that $\mathcal O_X(D)$ is $1$-away ACM.\
		
		\item $D=2D_1$, where $D_1$ is a line on $X$.  Applying \ref{RR} to $\mathcal O_X(2D_1), \mathcal O_X(2D_1-H), \mathcal O_X(2D_1-2H)$, we get $h^1(\mathcal O_X(2D_1)) \neq 0, h^1(\mathcal O_X(2D_1-H)) \neq 0, h^1(\mathcal O_X(2D_1-2H)) \neq 0$.\
		
		First, we show that $h^1(\mathcal O_X(2D_1+tH))=0$ for all $t\geq 1$. Note that by \ref{RR}, it is enough to show $h^0(\mathcal O_X(2D_1+tH)) \leq  \frac{3t^2+7t}{2} $ for all $t \geq 1$. By Lemma \ref{L1}, $(i)$, if we can show this for $t=1$, then by induction we are done. For $t=0$, $h^0(\mathcal O_X(2D_1+H))=5$ can be achieved by repeatedly tensorizing the structure exact sequence for the line $D_1$ with appropriate line bundles and then taking cohomology.\
		
		Next, we  show that $h^1(\mathcal O_X(D-tH))=0$ for all $t \geq 3$. By \ref{RR}, it is enough to show $h^0(\mathcal O_X((t-1)H-D))\leq  \frac{3t^2-7t}{2} $ for $t \geq 3$. By  Lemma \ref{L1}, $(ii)$, if we can show this for $t=3$, then by induction we are through. Note that for $t=3$, $h^0(\mathcal O_X(2H-2D_1)) \geq 3$ is there by \ref{RR}. Then the fact $h^0(\mathcal O_X(2H-2D_1)) = 3$ can be achieved by applying Proposition \ref{P2}, $(iv)$ to $\Gamma \in |2H-2D_1| \neq \emptyset$. This shows that $\mathcal O_X(D)$ is $3$-away ACM.\\
		
		\end{itemize}
	
	Note that in the above discussion we fixed the degree of the divisor $D$ on $X$ and determined the integer $\ell$ for which $\mathcal O_X(D)$ is $\ell$-away ACM. In the upcoming sections our approach of the study will be on the other direction i.e. we first fix the integer $\ell \geq 1$ for which $\mathcal O_X(D)$ is $\ell$-away ACM and then study the possibilities of the degree of $D$ in terms of $\ell$.\

	\section{Classification of $1$-away and $2$-away ACM line bundles}\label{cl12}
	
	Let $X$ be a nonsingular cubic hypersurface in $\mathbb P^3$ as before. In this section, our main aim is to classify $\ell$-away ACM line bundles on $X$ when $\ell=1,2$. Towards this, we organize this section as follows: in the first subsection i.e. in $\S \ref{cl1}$, we classify the $\ell=1$ case (Theorem \ref{T31}) with an explicit list of divisors (Proposition \ref{P32}). In the second subsection i.e. in $\S \ref{cl2}$, as an application of our degree bound result for $\ell (\geq 2)$-away ACM line bundles (Proposition \ref{degreebound}), we obtain a complete characterization of $2$-away ACM line bundles (Theorem \ref{T35}) again with an explicit list of divisors (Proposition \ref{P37}). We then apply these characterizations to obtain the classification of weakly Ulrich line bundles on $X$ that is not Ulrich (Corollary \ref{notUlrich}). \
	
	
	
	
	
	\subsection{Classification of $1$-away ACM line bundles}\label{cl1}\
	
	In this subsection, we prove the first main theorem of this paper i.e. a complete classification of initialized and $1$-away ACM line bundles on $X$.\
	
	\begin{theorem}\label{T31}
		Let $D$ be a non-zero effective divisor on a nonsingular cubic hypersurface $X \subset \mathbb P^3$. Then the following conditions are equivalent: 
		
		$(a)$ $\mathcal O_X(D)$ is initialized  and $1$-away ACM with respect to $H:= \mathcal O_{\mathbb P^3}(1)|_X$.\
		
		$(b)$ either $(D^2, D.H)=(-2,2)$ or $(D^2, D.H)=(2,4)$.

	\end{theorem}
	
	\begin{proof}
	
	\underline{\textbf{(a) $\implies$ (b)}}

	Let $D$ be a non-zero effective divisor on $X$. Assume $(a)$. Tensorizing the S.E.S \ref{ss} (with $\Gamma_1$ replaced by $H$) with $\mathcal O_X(H-D)$ and taking cohomology, one obtains $h^0(\mathcal O_X(H-D)) \leq h^0(\mathcal O_H(H-D))$. If $D.H >3$, then $h^0(\mathcal O_X(H-D))=0$. If $D.H=3$, then $h^0(\mathcal O_X(H-D)) =0$ (otherwise, if $|H-D| \neq \emptyset$, then $\mathcal O_X(D) \cong \mathcal O_X(H)$, a contradiction as $\mathcal O_X(D)$ is initialized). If $D.H=2$, then by Riemann-Roch theorem for line bundles on curves, one obtains $h^0(\mathcal O_H(H-D))=3-D.H =1$ and thus $h^0(\mathcal O_X(H-D)) \leq 1$. Note that $D.H \neq 1$ by Theorem \ref{Ton}, $(3)$. Therefore, we have established that in any case, $h^0(\mathcal O_X(H-D)) \leq 1$. We separately analyse both the situations  $h^0(\mathcal O_X(H-D)) =1$ and $h^0(\mathcal O_X(H-D))=0$.\\

	
	\pagebreak
	
	\underline{\textbf{Case $1$: $h^0(\mathcal O_X(H-D)) = 1$}}
	
	Observe that from the discussion in the previous paragraph it is clear that the only possibility in this situation is $D.H=2$. Then from the discussion on the $\ell$-away ACM line bundles of low degree on $X$ in subsection $\S \ref{12}$, it is clear that $D=D_1+D_2$ such that $D_i$'s are lines on $X$ with $D_1.D_2=0$. This gives us $(D^2,D.H)=(-2,2)$. Let $\Gamma \in |H-D| \neq \emptyset$. Since $\Gamma.H=1$, we have $\Gamma^2=-1$, a contradiction as in this case $(H-D)^2=-3$. Therefore, this case i.e. $h^0(\mathcal O_X(H-D)) = 1$ can't occur.\\
	
	\underline{\textbf{Case $2$: $h^0(\mathcal O_X(H-D)) = 0$}}
	
    Note that by Theorem \ref{Ton}, $(2)$, we have only two possibilities as follows:
    \begin{itemize}
    	\item either $h^1(\mathcal O_X(D-2H))\neq 0$ and $h^1(\mathcal O_X(D+tH))=0$ for $t \in \mathbb Z \setminus \{-2\}$.\
    	
    	\item or $h^1(\mathcal O_X(D-H))\neq 0$ and $h^1(\mathcal O_X(D+tH))=0$ for $t \in \mathbb Z \setminus \{-1\}$.\
    \end{itemize}
	In what follows next, we analyse each possibilities separately. Our strategy of the analysis is simple: if $h^1(\mathcal O_X(D+kH)) \neq 0$ for some $k \in \mathbb Z$, then we apply Riemann-Roch theorem (i.e. \ref{RR}) on the line bundles $\mathcal O_X(D+(k-1)H), \mathcal O_X(D+kH), \mathcal O_X(D+(k+1)H)$ and analyse the intersection theoretic data.\\
	
	Let us consider the first possibility i.e. $h^1(\mathcal O_X(D-2H))\neq 0$ and $h^1(\mathcal O_X(D+tH))=0$ for $t \in \mathbb Z \setminus \{-2\}$. Applying \ref{RR} on $\mathcal O_X(D-3H), \mathcal O_X(D-2H)$ and $\mathcal O_X(D-H)$, we get the following:
	
	\begin{itemize}
		\item $h^0(\mathcal O_X(2H-D))=\frac{D^2-5D.H+18}{2}+1$.
		
		\item $-h^1(\mathcal O_X(D-2H))= \frac{D^2-3D.H+6}{2}+1$.
		
		\item $D^2 =D.H-2$.
	\end{itemize}
	
	Combining the first and the third bullet point, one has $D.H \leq 4$. Combining the second and the third bullet point, one has $D.H \geq 4$. This forces $(D^2, D.H)=(2,4)$.\\
	

	Let us consider the second possibility i.e. $h^1(\mathcal O_X(D-H))\neq 0$ and $h^1(\mathcal O_X(D+tH))=0$ for $t \in \mathbb Z \setminus \{-1\}$. Applying \ref{RR}, on $\mathcal O_X(D-2H), \mathcal O_X(D-H)$ and $\mathcal O_X(D)$, we get the following :
	
	\begin{itemize}
		\item $D^2=3D.H-8$.\ 
		
		\item $D^2 <D.H-2$.
		
		\item $D^2+D.H \geq 0$.
	\end{itemize}

Combining the first and the third bullet point, one has $D.H \geq 2$. Combining the first and the second bullet point, one has $D.H \leq 2$. This forces $(D^2, D.H)=(-2,2)$.\\\

  


  
  	\underline{\textbf{(b) $\implies$ (a)}}
  	
  Let $D$ be a non-zero effective divisor on $X$. Assume $(b)$. Then either $(D^2, D.H)=(-2,2)$ or $(D^2, D.H)=(2,4)$. We separately show that in each case $\mathcal O_X(D)$ is $1$-away ACM (w.r.to $H$).\\
  
  \underline{\textbf{$(D^2, D.H)=(-2,2)$ }}
  
  This is already discussed in the last part of section $\S \ref{tech}$ (see the discussion in subsection $\S \ref{12}$, on $\ell$-away ACM line bundles of low degree  on $X$, the case $D.H=2$, third bullet point).\\
  
  
  
  
  \pagebreak
  
  \underline{\textbf{$(D^2, D.H)=(2,4)$ }}
  
  If we assume $|D-H| \neq \emptyset$, then there exists a line $\Gamma \in |D-H|$ and hence  $\Gamma^2=-1$, a contradiction as $\Gamma^2=(D-H)^2=-3$. Therefore, $h^0(\mathcal O_X(D-H))=0$ and $\mathcal O_X(D)$ is initialized.  Applying \ref{RR} to $\mathcal O_X(D-2H)$ and $\mathcal O_X(D-H)$ we get $h^1(\mathcal O_X(D-2H)) = 1$ and $h^1(\mathcal O_X(D-H)) = 0$ respectively. Therefore, it is sufficient to show $h^1(\mathcal O_X(D+tH)) =0$ for all $t \in \mathbb Z \setminus \{-2,-1\}$. Again as before, we divide this into two separate cases and use induction to establish both the cases:\
  
  
   First we  establish that $H^1(\mathcal O_X(D+tH)) = 0$ for $t\geq 0$. Note that by \ref{RR}, it is enough to show $h^0(\mathcal O_X(D+tH))\leq  \frac{3t^2+11t+8}{2} $ for $t\geq 0$. By Lemma \ref{L1}, $(i)$, if we can show this for $t=0$, then by induction we are done. For $t=0$, this can be seen by tensorizing S.E.S \ref{ss} (with $\Gamma_1$ replaced by $H$) by $\mathcal O_X(D)$ and taking cohomology.\
   
   
  
  Now we are left to prove that $H^1(\mathcal O_X(D-tH)) = 0$ for $t \geq 3$. By \ref{RR}, it is enough to show $h^0(\mathcal O_X((t-1)H-D))\leq  \frac{3t^2-11t+8}{2} $ for $t \geq 3$. By Lemma \ref{L1}, $(ii)$, if we can show this for $t=3$, then by induction we are through. Note that for $t=3$, the above is true by observing that  $h^0(\mathcal O_X(2H-D)) \geq 1$ and applying  Proposition \ref{P2}, $(i)$, to $\Gamma \in |2H-D|$.
   
  \end{proof}

\vspace{-62mm}
  
  Next, we proceed to give an explicit list of divisors on $X$ appearing in the above classification. Our strategy of obtaining the list is as follows: If $(D^2,D.H)=(-2,2)$, then letting $E=D+H$, one can see that $\mathcal O_X(E)$ is $0$-regular and therefore, by Theorem \ref{CM} globally generated and hence nef. Using Theorem \ref{cubic} and Theorem \ref{N}, one can then find out the list for $E$ by writing the possible tuples $(a,b_1,...,b_6)$ such that $a,b_i \geq 0$, $3a-\sum_{i=1}^{6} b_i=5, a^2-\sum_{i=1}^{6} b_i^2=5, a\geq b_i+b_j$ for $i \neq j$ and $2a \geq \sum_{i\neq j} b_i$ for each $j \in \{1,...6\}$. Once the list of $E$ is prepared, the list of $D$ can be obtained by using the expression of $H$ in terms of $l$ and the exceptional divisors.\\
  
  On the other hand, if $(D^2,D.H)=(2,4)$ it can be realized that $\mathcal O_X(D)$ itself is $0$-regular and therefore by theorem \ref{CM} globally generated and hence nef. Using thorem \ref{cubic} and theorem \ref{N}, one can then find out the list for $D$ by writing the possible tuples $(a,b_1,...,b_6)$ such that $a,b_i \geq 0$, $3a-\sum_{i=1}^{6} b_i=4, a^2-\sum_{i=1}^{6} b_i^2=2, a\geq b_i+b_j$ for $i \neq j$ and $2a \geq \sum_{i\neq j} b_i$ for each $j \in \{1,...6\}$.\\
  
  We summarize the above discussion in the following Proposition.\
  
  \vspace{-23mm}
  
\begin{proposition}\label{P32}
 Upto permutation of exceptional divisors on $X$, the initialized, $1$-away ACM divisors (w.r.to  $H:= \mathcal O_{\mathbb P^3}(1)|_X$) on $X$  are listed as follows:
 
 \vspace{-21mm}
 
  \begin{center}
  	\renewcommand{\arraystretch}{1.3}
  	\begin{tabular}{ |p{1.8cm}|p{5.5cm}|p{2cm}|  }
  		\hline
  		
  		\hline
  		$(D^2, D.H)$ & $D$  & $u(D)$ \\
  		\hline
  		& $e_5+e_6$ & $\binom{6}{2}$\\
  		& $l-e_1-e_2+e_6$  & $\binom{6}{3} \cdot \binom{3}{2}$\\
  		& $2l-e_1-e_2-e_3-e_4-e_5+e_6$  & $6$\\
  	$(-2,2)$	&$2l-2e_1-e_2-e_3$&  $\binom{6}{3} \cdot \binom{3}{2}$\\
  		   &$3l-2e_1-2e_2-e_3-e_4-e_5$ & $6 \cdot \binom{5}{2}$ \\
  		& $4l-2e_1-2e_2-2e_3-2e_4-e_5-e_6$ &  $\binom{6}{2}$\\
  		
  		\hline
  	\end{tabular}\
  \end{center}

\vspace{-13mm}

\pagebreak

\begin{center}
	\renewcommand{\arraystretch}{1.3}
	\begin{tabular}{ |p{1.8cm}|p{5.5cm}|p{2cm}|  }
		\hline
		
		\hline
		$(D^2, D.H)$ & $D$  & $u(D)$ \\
  		\hline
  		& $2l-e_1-e_2$ & $\binom{6}{2}$\\
  		& $3l-2e_1-e_2-e_3-e_4$   &  $4\binom{6}{4}$\\
  	$(2,4)$	&$4l-3e_1-e_2-e_3-e_4-e_5-e_6$ & $6$\\
  		  &$4l-2e_1-2e_2-2e_3-e_4-e_5$ & $6 \binom{5}{2}$ \\
  		& $5l-3e_1-2e_2-2e_3-2e_4-e_5-e_6$ & $4 \binom{6}{2}$ \\
  		& $6l-3e_1-3e_2-2e_3-2e_4-2e_5-2e_6$ & $\binom{6}{2}$   \\
  		
  		\hline
  	\end{tabular}\
  \end{center}
  
 where $u(D)$ is the number of divisors obtained permuting the exceptional divisors in the writing of $D$. 
\end{proposition} 


We end this subsection with the following remarks:

\begin{remark}\normalfont
	$(i)$ From  \cite{BPP}, it follows that for any finite linear projection $\pi :X \to \mathbb P^2$, $\pi_*(\mathcal O_X(D)) \cong \Omega^1_{\mathbb P^2}(1) \oplus \mathcal O_{\mathbb P^2}$ if $(D^2, D.H)=(-2,2)$ and $\pi_*(\mathcal O_X(D))$ is isomorphic to either $\Omega^1_{\mathbb P^2}(2) \oplus \mathcal O_{\mathbb P^2}$ or $\Omega^1_{\mathbb P^2}(2) \oplus \mathcal O_{\mathbb P^2}(-1)$ if $(D^2, D.H)=(2,4)$. Moreover, it can also be seen from \cite{BPP} that not every rank $3$, $1$-away ACM bundle on $\mathbb P^2$ can be obtained as a finite linear projection of a $1$-away ACM line bundle  on $X$.\
	
	$(ii)$ Any  line bundle as in Theorem \ref{T31} is weakly Ulrich.\\
	
	
\end{remark}

	\subsection{Classification of $2$-away ACM line bundles}\label{cl2}\
	
	 As an application of Proposition \ref{degreebound}, in this subsection, we give a proof of the second main theorem of this article i.e. a complete classification of $2$-away ACM line bundles on a nonsingular cubic surface $X$. Before proceeding for the classification theorem, we prepare a lemma to identify the possible twists for which the intermediate cohomology of a $2$-away ACM line bundle of certain degrees on $X$ do not vanish.\\

	\begin{lemma}\label{2awaynecessary}
		Let $D$ be a non-zero effective divisor on $X$ such that $3 \leq D.H \leq 6$ and $\mathcal O_X(D)$ is initialized and $2$-away ACM. Then one of  the following occurs:
		
		$(i)$ If $h^1(\mathcal O_X(D-H)) \neq 0$, then $h^1(\mathcal O_X(D-2H)) \neq 0$.
		
		$(ii)$  If $h^1(\mathcal O_X(D-H)) = 0$ and $h^1(\mathcal O_X(D-2H)) \neq 0$, then $h^1(\mathcal O_X(D-3H)) \neq 0$.
		
	\end{lemma}
	
	\begin{proof}
		$(i)$ If $h^1(\mathcal O_X(D-H)) \neq 0$, then applying \ref{RR} to $\mathcal O_X(D-H)$, one obtains $D^2 \leq D.H-4$. Now, if on the contrary, we let $h^1(\mathcal O_X(D-2H)) = 0$, then applying \ref{RR} to $\mathcal O_X(D-2H)$, we get $D^2 \geq 3D.H-8$. This coupled with the first inequality give us $D.H \leq 2$, a contradiction.\
		
		$(ii)$ If $h^1(\mathcal O_X(D-H)) = 0$, then it is easy to see that $\mathcal O_X(D)$ is $0$-regular and hence by Theorem \ref{CM}, globally generated and $h^1(\mathcal O_X(D+tH)) =0$ for all $t \geq -1$. By Theorem \ref{Ton}, $(2)$, one has $h^1(\mathcal O_X(D-2H)) \neq 0$. If on the contrary, we assume  $h^1(\mathcal O_X(D-3H)) = 0$. Then again one can deduce that $\mathcal O_X(3H-D)$ is  $0$-regular and hence by Theorem \ref{CM}, globally generated and $h^1(\mathcal O_X(D-(4+t)H))=0$ for all $t \geq -1$, a contradiction as then $\mathcal O_X(D)$ becomes $1$-away ACM.\
		
	\end{proof}

Now we are ready to prove the $2$-away Classification theorem:

\begin{theorem}\label{T35}
	Let $D$ be a non-zero effective divisor on a nonsingular cubic hypersurface $X \subset \mathbb P^3$. Then the following conditions are equivalent: 
	
	$(a)$ $\mathcal O_X(D)$ is initialized  and $2$-away ACM with respect to $H:=\mathcal O_{\mathbb P^3}(1)|_X$.\
	
	$(b)$ One of the following cases occur:
	
	$(i)$ $(D^2, D.H)=(4,6)$,  $D$ is nef and $|3H-D|=\emptyset$.\
	
	$(ii)$ $(D^2, D.H)=(3,5)$, $D$ is nef.\
	
	$(iii)$ $(D^2, D.H)=(0,4)$, $|2H-D|=\emptyset$. \
	
	$(iv)$ $(D^2, D.H)=(-3,3)$.\
	
	$(v)$ $(D^2, D.H)=(-1,3)$.\

\end{theorem}

\begin{proof}
	
	\underline{\textbf{(a) $\implies$ (b)}}
	
	Let $D$ be a non-zero effective divisor on $X$. Assume $(a)$. Then by Proposition \ref{degreebound}, $(i)$, we must have $1 \leq D.H \leq 6$. Note that by Theorem \ref{Ton}, $(3)$, we have $D.H \neq 1$. By discussion in subsection $\S \ref{12}$, it follows that $D.H \neq 2$. Therefore, we must have $3 \leq D.H \leq 6$ and hence by Lemma \ref{2awaynecessary}, there are two possibilities in terms of $h^1$ nonvanishing. In what follows, we separately analyse both the situations.\\
	
	\underline{\textbf{Case-1: $h^1(\mathcal O_X(D-H)) \neq 0, h^1(\mathcal O_X(D-2H)) \neq 0$ and $h^1(\mathcal O_X(D+tH))=0$ for $t \in \mathbb Z \setminus \{-1,-2\}$}}.
	
	Applying \ref{RR} to the line bundles $\mathcal O_X(D), \mathcal O_X(D-H), \mathcal O_X(D-2H)$ and $\mathcal O_X(D-3H)$, we get the following:
	
	\begin{itemize}
		\item $D^2+D.H \geq 0$
		
		\item $D^2 \leq D.H-4$
		
		\item $D^2-3D.H+10 \leq 0$
		
		\item $D^2-5D.H+20 \geq 0$
	\end{itemize}
	
	Note that for $D.H=6$, the second bullet point implies $D^2 \leq 2$ and the fourth bullet point implies $D^2 \geq 10$, a contradiction. Similarly, for $D.H=5$, the second bullet point implies $D^2 \leq 1$ and the fourth bullet point implies $D^2 \geq 5$, a contradiction. For $D.H=4$, the second and fourth bullet point together implies $D^2=0$. So we have the possibility $(D^2, D.H)=(0,4)$ and in this case it is easy to see that $h^0(\mathcal O_X(2H-D))=0$ (thus we achieve $(b)(iii)$). For $D.H=3$, we have $-3 \leq D^2 \leq -1$. It is easy to see that in this case $D^2=-2$ is not possible, otherwise $p_a(D)$ is not an integer. Therefore, we end up having two more possibilities $(D^2, D.H)=(-1,3)$ or $(D^2, D.H)=(-3,3)$ (thus we achieve $(b)(iv), (b)(v)$).\\
	
	\underline{\textbf{Case-2: $h^1(\mathcal O_X(D-2H)) \neq 0, h^1(\mathcal O_X(D-3H)) \neq 0$ and $h^1(\mathcal O_X(D+tH))=0$ for $t \in \mathbb Z \setminus \{-2,-3\}$}}
	
	Applying \ref{RR} to the line bundles $\mathcal O_X(D-H)$ and $\mathcal O_X(D-2H)$, we get the following:
	\begin{itemize}
		\item $D^2=D.H-2$.\
		
		\item $D^2-3D.H+10 \leq 0$.\
	\end{itemize}
	
	Applying \ref{RR} to the line bundle $\mathcal O_X(D-3H)$, we have $h^0(\mathcal O_X(2H-D)) \geq  \frac{D^2-5D.H+18}{2}+2$. Applying cohomology to the S.E.S
	\begin{align}
	 0 \to \mathcal O_X(H-D) \to \mathcal O_X(2H-D) \to \mathcal O_H(2H-D) \to 0,
	 \end{align}
	 
	 we have $h^0(\mathcal O_X(2H-D)) \leq h^0(\mathcal O_H(2H-D)) \leq 3$ for $D.H \in \{3,4,5\}$. For $D.H=6$, $h^0(\mathcal O_X(2H-D))=0$. Therefore, in any case $h^0(\mathcal O_X(2H-D)) \leq 3$.  This means :
	\begin{itemize}
		\item $D^2-5D.H+16 \leq 0$
	\end{itemize}
	Applying \ref{RR} to the line bundle $\mathcal O_X(D-4H)$, we have 
	\begin{itemize}
		\item $D^2-7D.H+38 \geq 0$
	\end{itemize}
	Note that for $D.H=3$, we have from the first bullet point $D^2=1$, a contradiction to the second bullet point. Observe that,  for $D.H=4$, we have from the first bullet point $D^2=2$. Therefore, in this case one possibility is $(D^2, D.H)=(2,4)$. This is not possible because then by Theorem \ref{T31}, $\mathcal O_X(D)$ is $1$-away ACM. Note that for $D.H=5$, we have from the first bullet point $D^2=3$. Therefore, in this case another possibility is $(D^2, D.H)=(3, 5)$ and by Lemma \ref{nef}, $D$ is nef (hence we obtain $(b) (ii)$). Note that for $D.H=6$, we have from the first bullet point $D^2=4$. Therefore, in this case another possibility is $(D^2, D.H)=(4, 6)$, $D$ is nef (by Lemma \ref{nef}) and by \ref{RR}, $h^0(\mathcal O_X(3H-D))=0$ (hence we obtain $(b)(i)$).\\
	
	
	
	\underline{\textbf{(b) $\implies$ (a)}}
	
 Let	$D$ be a non-zero effective divisor on $X$. Assume $(b)$. We will analyze the five cases separately.\
 
 
 \underline{\textbf{$(b)(i) : (D^2, D.H)=(4,6)$, $D$ is nef, $|3H-D|= \emptyset$ }}
 
 From Lemma \ref{nef}, we have $\mathcal O_X(D)$ is initialized. Note that applying \ref{RR} to $\mathcal O_X(D-2H)$, we have $h^1(\mathcal O_X(D-2H))=3$. It is easy to see that $h^0(\mathcal O_X(2H-D)) =0$ and hence applying \ref{RR} to $\mathcal O_X(D-3H)$, we have $h^1(\mathcal O_X(D-3H))=3$.  Therefore, it is enough to show that $h^1(\mathcal O_X(D+tH)) =0$ for all $t \in \mathbb Z \setminus \{-2,-3\}$. We divide this into two separate cases and use induction to establish both the cases:\
 \
 
 First, we  show that $h^1(\mathcal O_X(D+tH))=0$ for all $t \geq -1$.  Note that by \ref{RR}, it is enough to show that $h^0(\mathcal O_X(D+tH))\leq  \frac{3t^2+15t+12}{2} $ for all $t \geq -1$. By Lemma \ref{L1}, $(i)$, if we can show this for $t=-1$, then by induction we are done. For $t=-1$, this is indeed true by	 initializedness.\


 
 Next, we  show that $h^1(\mathcal O_X(D-tH))=0$ for all $t \geq 4$.  By \ref{RR}, it is enough to show $h^0(\mathcal O_X((t-1)H-D))\leq  \frac{3t^2-15t+12}{2} $ for $t =4$. By  Lemma \ref{L1}, $(ii)$, if we can show this for $t=4$, then by induction we are through. Note that for $t=4$, we have this from the assumption $|3H-D|=\emptyset$.\\

 
 \underline{\textbf{$(b)(ii) : (D^2, D.H)=(3,5)$, $D$ is nef}}
 
 From Lemma \ref{nef}, we have $\mathcal O_X(D)$ is initialized. Applying \ref{RR} to $\mathcal O_X(D-2H)$, we get $h^1(\mathcal O_X(D-2H))=2$. Next, note that $h^0(\mathcal O_X(2H-D)) =0$, otherwise if $\Gamma \in |2H-D|$, then $\Gamma.H=1$ and hence $\Gamma^2=-1$, but $(2H-D)^2=-5$, a contradiction. Then applying \ref{RR} to $\mathcal O_X(D-3H)$, we get $h^1(\mathcal O_X(D-3H))=1$. Therefore, it is enough to show that $h^1(\mathcal O_X(D+tH)) =0$ for all $t \in \mathbb Z \setminus \{-2,-3\}$. We divide this into two separate cases and use induction to establish both the cases:\

  First, we show that $h^1(\mathcal O_X(D+tH))=0$ for all $t \geq -1$.  Note that by \ref{RR}, it is enough to show $h^0(\mathcal O_X(D+tH))\leq  \frac{3t^2+13t+10}{2} $ for all $t \geq -1$. By Lemma \ref{L1}, $(i)$, if we can show this for $t=-1$, then by induction we are done. For $t=-1$, this is indeed true by	 initializedness.\
  
  Next, we show that $h^1(\mathcal O_X(D-tH))=0$ for all $t \geq 4$.  By \ref{RR}, it is enough to show $h^0(\mathcal O_X((t-1)H-D))\leq  \frac{3t^2-13t+10}{2} $ for $t =4$. By  Lemma \ref{L1}, $(ii)$, if we can show this for $t=4$, then by induction we are through. Note that for $t=4$, $h^0(\mathcal O_X(3H-D))=3$ follows from Proposition \ref{P2}, $(iv)$.\\
  
 
  \underline{\textbf{$(b)(iii) : (D^2, D.H)=(0,4)$, $|2H-D| = \emptyset$ }}
  
  
   Let $\Gamma \in |D-H|$. Then $\Gamma$ is a line and hence $\Gamma^2=-1$. But $(D-H)^2=-5$, a contradiction. Therefore, $\mathcal O_X(D)$ is initialized.  Applying \ref{RR} to $\mathcal O_X(D-H)$, we get $h^1(\mathcal O_X(D-H))=1$. Next, applying \ref{RR} to $\mathcal O_X(D-2H)$, we get $h^1(\mathcal O_X(D-2H))=2$. Therefore, it is enough to show that $h^1(\mathcal O_X(D+tH)) =0$ for all $t \in \mathbb Z \setminus \{-1,-2\}$. We divide this into two separate cases and use induction to establish both the cases:\

  
  
  First, we  show that $h^1(\mathcal O_X(D+tH))=0$ for all $t \geq 0$. Note that by \ref{RR}, it is enough to show $h^0(\mathcal O_X(D+tH))\leq  \frac{3t^2+11t+6}{2} $ for all $t \geq 0$. By Lemma \ref{L1}, $(i)$, if we can show this for $t=0$, then by induction we are done. For $t=0$, $h^0(\mathcal O_X(D))=3$ follows from Proposition \ref{P2}, $(iv)$.\


  Next, we show that $h^1(\mathcal O_X(D-tH))=0$ for all $t \geq 3$. By \ref{RR}, it is enough to show $h^0(\mathcal O_X((t-1)H-D))\leq  \frac{3t^2-11t+6}{2} $ for $t \geq 3$. By  Lemma \ref{L1}, $(ii)$, if we can show this for $t=3$, then by induction we are through. Note that for $t=3$, we have this from the assumption $|2H-D|=\emptyset$.\\
  
  
  \underline{\textbf{$(b)(iv) : (D^2, D.H)=(-3,3)$ }}
  
  	 If $|D-H| \neq \emptyset$, then we have $\mathcal O_X(D) \cong \mathcal O_X(H)$, which gives us $D^2 =3$, a contradiction. Therefore,  $\mathcal O_X(D)$ is initialized.  Applying \ref{RR} to $\mathcal O_X(D-H)$, we get $h^1(\mathcal O_X(D-H))=2$. Next, applying \ref{RR} to $\mathcal O_X(D-2H)$, we get $h^1(\mathcal O_X(D-2H))=2$. Therefore, it is enough to show that $h^1(\mathcal O_X(D+tH)) =0$ for all $t \in \mathbb Z \setminus \{-1,-2\}$. We divide this into two separate cases and use induction to establish both the cases:\
  	 

  
  First,  we show that $h^1(\mathcal O_X(D+tH))=0$ for all $t \geq 0$. Note that by \ref{RR}, it is enough to show $h^0(\mathcal O_X(D+tH))\leq  \frac{3t^2+9t+2}{2} $ for all $t \geq 0$. By Lemma \ref{L1}, $(i)$, if we can show this for $t=0$, then by induction we are done. For $t=0$, $h^0(\mathcal O_X(D))=1$ follows from Proposition \ref{P2}, $(ii)$.\


  Next, we  show that $h^1(\mathcal O_X(D-tH))=0$ for all $t \geq 3$. By \ref{RR}, it is enough to show $h^0(\mathcal O_X((t-1)H-D))\leq  \frac{3t^2-9t+2}{2} $ for $t \geq 3$. By  Lemma \ref{L1}, $(ii)$, if we can show this for $t=3$, then by induction we are through. Note that for $t=3$, $h^0(\mathcal O_X(2H-D))=1$ follows from Proposition \ref{P2}, $(ii)$.\\

  \underline{\textbf{$(b)(v) : (D^2, D.H)=(-1,3)$ }}
  
    If $|D-H| \neq \emptyset$, then we have $\mathcal O_X(D) \cong \mathcal O_X(H)$, which gives us $D^2 =3$, a contradiction. Therefore, $\mathcal O_X(D)$ is initialized. Applying \ref{RR} to $\mathcal O_X(D-H)$, we get $h^1(\mathcal O_X(D-H))=1$. Next, applying \ref{RR} to $\mathcal O_X(D-2H)$, we get $h^1(\mathcal O_X(D-2H))=1$. Therefore, it is enough to show that $h^1(\mathcal O_X(D+tH)) =0$ for all $t \in \mathbb Z \setminus \{-1,-2\}$. We divide this into two separate cases and use induction to establish both the cases:\
    
  
  First, we  show that $h^1(\mathcal O_X(D+tH))=0$ for all $t \geq 0$. Note that by \ref{RR}, it is enough to show $h^0(\mathcal O_X(D+tH))\leq  \frac{3t^2+9t+4}{2} $ for all $t \geq 0$. By Lemma \ref{L1}, $(i)$, if we can show this for $t=0$, then by induction we are done. For $t=0$, $h^0(\mathcal O_X(D))=2$ follows from Proposition \ref{P2}, $(iii)$.\
  

  Next, we  show that $h^1(\mathcal O_X(D-tH))=0$ for all $t \geq 3$. By \ref{RR}, it is enough to show $h^0(\mathcal O_X((t-1)H-D))\leq  \frac{3t^2-9t+4}{2} $ for $t \geq 3$. By  Lemma \ref{L1}, $(ii)$, if we can show this for $t=3$, then by induction we are through. Note that for $t=3$, $h^0(\mathcal O_X(2H-D))=2$ follows from Proposition \ref{P2}, $(iii)$.\

\end{proof}

Next, we proceed to give an explicit list of divisors appearing in the above classification. Our strategy of obtaining the list is as follows:

\begin{itemize}
	
\item If $(D^2,D.H)=(4,6)$ and $D$ is nef,  then using Theorem \ref{cubic} and Theorem \ref{N}, one can  find out the list for $D$ by writing the possible tuples $(a,b_1,...,b_6)$ such that $a,b_i \geq 0$, $3a-\sum_{i=1}^{6} b_i=6, a^2-\sum_{i=1}^{6} b_i^2=4, a\geq b_i+b_j$ for $i \neq j$ and $2a \geq \sum_{i\neq j} b_i$ for each $j \in \{1,...6\}$. Next, for each such $D$,  we look at the expression $3H-D$. If it is of the form $D_1+2D_2$ for two non-intersecting lines $D_1,D_2$ on $X$, then we discard the possibility (because then $h^0(\mathcal O_X(3H-D)) \neq 0$). For each of the remaining cases we see that $4H-D$ is nef and big. Using Kawamata-Viehweg vanishing, we get $h^0(\mathcal O_X(3H-D)) = h^{1}(\mathcal{O}_{X}(D-4H))=0$ and we obtain the complete list.\

\item If $(D^2,D.H)=(3,5)$ and $D$ is nef, then using Theorem \ref{cubic} and Theorem \ref{N}, one can  find out the list for $D$ by writing the possible tuples $(a,b_1,...,b_6)$ such that $a,b_i \geq 0$, $3a-\sum_{i=1}^{6} b_i=5, a^2-\sum_{i=1}^{6} b_i^2=3, a\geq b_i+b_j$ for $i \neq j$ and $2a \geq \sum_{i\neq j} b_i$ for each $j \in \{1,...6\}$.\

\item If $(D^2,D.H)=(0,4)$, then letting $E=D+H$ one can see that $\mathcal O_X(E)$ is $0$-regular and therefore by Theorem \ref{CM}, globally generated and hence nef. Then using Theorem \ref{cubic} and Theorem \ref{N}, one can then find out the list for $E$ by writing the possible tuples $(a,b_1,...,b_6)$ such that $a,b_i \geq 0$, $3a-\sum_{i=1}^{6} b_i=7, a^2-\sum_{i=1}^{6} b_i^2=11, a\geq b_i+b_j$ for $i \neq j$ and $2a \geq \sum_{i\neq j} b_i$ for each $j \in \{1,...6\}$. Once the list of $E$ is prepared, the list of $D$ can be obtained by using the expression of $H$ in terms of $l$ and the exceptional divisors. Next, for each such $D$, we look at the expression $2H-D$. If it is of the form $2D_1$ for a line $D_1$ on $X$, then we discard the possibility (because then $h^0(\mathcal O_X(2H-D)) \neq 0$). For each of the remaining cases we see that $3H-D$ is nef and big. Using Kawamata-Viehweg vanishing, we get $h^0(\mathcal O_X(2H-D)) = h^{1}(\mathcal{O}_{X}(D-3H))=0$ and we obtain the complete list.\

\item If $(D^2,D.H)=(-3,3)$, then letting $E=D+H$, one can see that $\mathcal O_X(E)$ is $0$-regular and therefore by Theorem \ref{CM}, globally generated and hence nef. Using Theorem \ref{cubic} and Theorem \ref{N}, one can then find out the list for $E$ by writing the possible tuples $(a,b_1,...,b_6)$ such that $a,b_i \geq 0$, $3a-\sum_{i=1}^{6} b_i=6, a^2-\sum_{i=1}^{6} b_i^2=6, a\geq b_i+b_j$ for $i \neq j$ and $2a \geq \sum_{i\neq j} b_i$ for each $j \in \{1,...6\}$. Once the list of $E$ is prepared, the list of $D$ can be obtained by using the expression of $H$ in terms of $l$ and the exceptional divisors.\

\item If $(D^2,D.H)=(-1,3)$, then letting $E=D+H$, one can see that $\mathcal O_X(E)$ is $0$-regular and therefore by Theorem \ref{CM}, globally generated and hence nef. Using Theorem \ref{cubic} and Theorem \ref{N}, one can then find out the list for $E$ by writing the possible tuples $(a,b_1,...,b_6)$ such that $a,b_i \geq 0$, $3a-\sum_{i=1}^{6} b_i=6, a^2-\sum_{i=1}^{6} b_i^2=8, a\geq b_i+b_j$ for $i \neq j$ and $2a \geq \sum_{i\neq j} b_i$ for each $j \in \{1,...6\}$. Once the list of $E$ is prepared, the list of $D$ can be obtained by using the expression of $H$ in terms of $l$ and the exceptional divisors.\\
\end{itemize}

\vspace{-35mm}

We summarize the above discussion in the following Proposition.\\

\vspace{-44mm}

\begin{proposition}\label{P37}
	Upto permutation of exceptional divisors on $X$, the initialized, $2$-away ACM divisors (w.r.to  $H:= \mathcal O_{\mathbb P^3}(1)|_X$) on $X$  are listed as follows:
	
	\vspace{-33mm}
	\begin{center}
		\renewcommand{\arraystretch}{1.3}
		\begin{tabular}{ |p{2.2cm}|p{5.7cm}|p{2cm}|  }

			\hline
			$(D^2, D.H)$ & $D$  & $u(D)$ \\
			\hline
			$(4,6)$+ & $2l$ & $1$\\
			$D$ nef + & $6l-4e_1-2e_2-2e_3-2e_4-2e_5$ &  $\binom{6}{5} \cdot \binom{5}{4}$\\
			$|3H-D| =\emptyset$ & $10l-4e_1-4e_2-4e_3-4e_4-4e_5-4e_6$ & $1$ \\
			
%
%
			\hline
		\end{tabular}
	\end{center}

\vspace{-23mm}

\pagebreak

\begin{center}
	\renewcommand{\arraystretch}{1.3}
	\begin{tabular}{ |p{2.2cm}|p{5.7cm}|p{2cm}|  }

		\hline
		$(D^2, D.H)$ & $D$  & $u(D)$ \\
		
		\hline
		
		& $2l-e_1$ & $6$\\
		& $3l-2e_1-e_2-e_3$   & $\binom{6}{3}\cdot \binom{3}{2}$ \\
		&$4l-3e_1-e_2-e_3-e_4-e_5$ & $\binom{6}{5} \cdot \binom{5}{4}$\\
		$(3,5)+$  &$4l-2e_1-2e_2-2e_3-e_4$ & $\binom{6}{2}$ \\
		$D$ nef & $6l-4e_1-2e_2-2e_3-2e_4-2e_5-e_6$ & $2\binom{6}{2}$\\
		& $6l-3e_1-3e_2-3e_3-2e_4-e_5-e_6$ & $3\binom{6}{3}$   \\
		& $7l-4e_1-3e_2-3e_3-2e_4-2e_5-2e_6$ & $3\binom{6}{3}$   \\
		& $8l-4e_1-3e_2-3e_3-3e_4-3e_5-3e_6$ &  $6$  \\
		\hline
		
		\hline
					& $l+e_6$ & $6$\\
					& $2l-e_1-e_2-e_3+e_6$   & $\binom{6}{2}$ \\
					&$3l-2e_1-2e_2-e_3$ & $3 \binom{6}{3}$\\
					$(0,4)+$  &$3l-2e_1-e_2-e_3-e_4-e_5+e_6$ & $6$ \\
					$|2H-D|=\emptyset$ & $4l-3e_1-2e_2-e_3-e_4-e_5$ & $6$\\
					& $5l-3e_1-3e_2-2e_3-e_4-e_5-e_6$ & $12 \binom{5}{3}$   \\
				& $6l-3e_1-3e_2-3e_3-2e_4-2e_5-e_6$ & $\binom{6}{3} \cdot \binom{3}{2}$  \\
					& $7l-3e_1-3e_2-3e_3-3e_4-3e_5-2e_6$ & $6$   \\
					
					\hline
			
			& $e_4+e_5+e_6$ & $\binom{6}{3}$\\
			& $l-e_1-e_2+e_5+e_6$   & $\binom{6}{4} \cdot \binom{4}{2}$\\
			& $2l-2e_1-e_2-e_3+e_6$ & $2 \binom{6}{4} \binom{4}{2}$\\
			$(-3,3)$  & $3l-3e_1-e_2-e_3-e_4$ & $4 \binom{6}{4}$ \\
			& $3l-2e_1-2e_2-e_3-e_4-e_5+e_6$ & $6 \binom{5}{3}$\\
			& $4l-3e_1-2e_2-2e_3-e_4-e_5-e_6$ & $3\binom{6}{3}$   \\
			& $5l-3e_1-3e_2-2e_3-2e_4-e_5-e_6$ &  $\binom{6}{2} \binom{4}{2}$  \\
			& $6l-3e_1-3e_2-3e_3-2e_4-2e_5-2e_6$ &  $\binom{6}{3}$  \\

			
			\hline
			& $l-e_1+e_6$ & $2 \binom{6}{2}$\\
			& $2l-2e_1-e_2$   & $2 \binom{6}{2}$ \\
			& $2l-e_1-e_2-e_3-e_4+e_6$ & $5 \binom{6}{5}$ \\
			$(-1,3)$  & $3l-2e_1-2e_2-e_3-e_4$ & $\binom{6}{2} \cdot \binom{4}{2}$ \\
			& $4l-3e_1-2e_2-e_3-e_4-e_5-e_6$ & $2 \binom{6}{2}$\\
			& $4l-2e_1-2e_2-2e_3-e_5$ & $4 \cdot \binom{6}{2}$   \\
			& $5l-3e_1-2e_2-2e_3-2e_4-2e_5-e_6$ &  $2 \binom{6}{4} $  \\

			\hline
	\end{tabular}
\end{center}


%
%
%
%
%
%

	where $u(D)$ is the number of divisors obtained permuting the exceptional divisors in the writing of $D$. 
\end{proposition} 

	\begin{remark}\normalfont
		
		
		$(i)$ Divisors appearing in Theorem \ref{T35}, $(b) (iii),(iv),(v) $ are weakly Ulrich and the remaining are not weakly Ulrich.\
		
		
		$(ii)$ We document a few observation about the case Theorem \ref{T35}, $(b) (i) $:
		
		\begin{itemize}
			\item No divisor with $D^{2} = 4$, and $D \cdot H=6$ is ample. Let $D = al - \sum_{i=1}^{6}b_{i}e_{i}$. Then we get $3a-\sum_{i=1}^{6}b_{i} = 6$, and $a^{2} - \sum_{i=1}^{6}b_{i}^{2} =4$. By classification of ample divisors on $X$, we need $b_{i} >0$, $a > b_{i} + b_{j}$ for any pair $i, j$, and $2a > \sum_{i}^{i \ne j} b_{i}$ for all $j$. These relations come from the intersection with lines classes $e_{i}$, $l-e_{i}-e_{j}$, and $l-\sum_{i}^{i \ne j} e_{i}$. Let $b_{6} = \text{min}\{b_{1}, b_{2}, \cdots b_{6}\}$. Then we have $2a > b_{1} + b_{2} + \cdots + b_{5}$ and $2a + b_{6} > 3a-6$. We have $a \le b_{6} + 5$. Thus, $b_{i} \le 4$ and $a \le 10$. Thus, atleast one of $b_{1}, \cdots, b_{5}$ is $\le 3$. So $b_{6} \le 3$. We also need $a \ge 4$. For $a =4$, $b_{i} =1$, and $a^{2} - \sum_{i=1}^{6}b_{i}^{2} =4$ is not satisfied. For $a=5$, $b_{i}$ can only be $2, 1$(atleast two of them is $1$) or $3, 1$(only one of them is $3$). Again $\sum_{i=1}^{6}b_{i}^{2} =21$ will not be satisfied. Similar arguments for other values of $a$. \
			
			\item But the collection contains nef divisors, for e.g., $D = 2l$. Also it contains non nef divisors e.g., $D = H + 2e_{1} +e_{2} = 3l + e_{1} -e_{3} - \cdots - b_{6}$. $D \cdot e_{1} =-1$.\
			
			\item If $D$ satisfies all the conditions of Theorem \ref{T35}, $(b) (i) $, then  $D$ is nef, and $D^{2} =4 >0$ implies that $D$ is big. Also $D$ is globally generated by Castelnuovo-Mumford regularity. But still $H^{0}(\mathcal{O}_{X}(3H-D)) $ may not be zero,  e.g., $D =9l-5E_{1}-4E_{2}-3E_{3}-3E_{4}-3E_{5}-3E_{6}$. Thus, geometrically the condition $H^{0}(\mathcal{O}_{X}(3H-D))=0$ is milder than $D$ being ample, yet stronger than $D$ being positive (nef, big, globally generated).\
			
			\item \label{ex1}
			Let $D = 9l-5e_{1}-4e_{2}-3e_{3}-3e_{4}-3e_{5}-3e_{6}$. Then $D$ is nef by Theorem \eqref{N}. So we will have $\mathcal O_X(D)$ is initialized, and from \ref{RR}, it follows that $H^{1}(\mathcal{O}_{X}(D-2H)) \ne 0$, and $H^{1}(\mathcal{O}_{X}(D-3H)) \ne 0$. Since $3H-D$ is effective, by \ref{RR}, we have $H^{1}(\mathcal{O}_{X}(3H-D)) \ne 0$ and hence by Serre duality $H^{1}(\mathcal{O}_{X}(D-4H)) \ne 0$. 
			We can check $5H-D$ is nef by Theorem \eqref{N}. The divisor $4H-D$ is effective. Since $5H-D$ is the sum of ample and effective divisor, it is big. By Kawamata-Viehweg vanishing we have $H^{1}(\mathcal{O}_{X}(D-5H))=0$. Since from initializedness $h^0(\mathcal O_X(D-H))=0$, from \ref{RR}, it follows that $h^1(\mathcal O_X(D-H))=0$. Therefore, by Lemma \ref{general}, $\mathcal O_X(D)$ is an initialized and $3$-away ACM line bundle.\

		\end{itemize}

	\end{remark}

As an application of our main theorems (i.e. Theorem \ref{T31} and Theorem \ref{T35}), in the following Corollary, we give a characterization of weakly Ulrich line bundles on $X$ that are not Ulrich.\

\begin{corollary}\label{notUlrich}
	Let $D$ be a non-zero effective divisor on a nonsingular cubic hypersurface $X \subset \mathbb P^3$. Then the following conditions are equivalent: 
	
	$(a)$ $\mathcal O_X(D)$ is initialized  and weakly Ulrich but not Ulrich (with respect to  $H:=\mathcal O_{\mathbb P^3}(1)|_X$).\
	
	$(b)$ One of the following occur:
	
	$(i)$ $(D^2,D.H)= (0,2)$.\
	
	$(ii)$ $(D^2,D.H)= (-1,1)$.\
	
	$(iii)$ $(D^2,D.H)= (-2,2)$.\
	
	$(iv)$ $(D^2,D.H)= (2,4)$.\
	
	$(v)$ $(D^2,D.H)= (0,4)$, $|2H-D| =\emptyset$.\
	
	$(vi)$ $(D^2,D.H)= (-3,3)$.\
	
	$(vii)$ $(D^2,D.H)= (-1,3)$.\
\end{corollary}

\begin{proof}
Let $D$ be a non-zero effective divisor on a nonsingular cubic hypersurface $X \subset \mathbb P^3$. Assume that  $\mathcal O_X(D)$ is initialized  and weakly Ulrich but not Ulrich, then from the definition for the surface case, we see that there are two possibilities:
\begin{itemize}
	\item either $h^1(\mathcal O_X(D-H)) \neq 0$ or  $h^1(\mathcal O_X(D-2H)) \neq 0$ (or both) in which case $\mathcal O_X(D)$ is either $1$ or $2$-away ACM. By Theorem \ref{T35}, we see that the  possibilities in the $2$-away ACM cases are  $(D^2,D.H)= (0,4)$ with $|2H-D| =\emptyset$ or $(D^2,D.H)= (-3,3)$ or  $(D^2,D.H)= (-1,3)$. By Theorem \ref{T31}, we see that the  possibilities in the $1$-away ACM cases are  $(D^2,D.H)= (-2,2)$ and $(D^2,D.H)= (2,4)$.\
	
	\item both $h^1(\mathcal O_X(D-H)) = h^1(\mathcal O_X(D-2H))= 0$ and $h^2(\mathcal O_X(D-2H)) \neq 0$. Note that in this case $\mathcal O_X(D)$ is ACM and therefore, by Theorem \ref{Ton}, $(3)$, the only two possibilities are $(D^2,D.H)= (0,2)$ and $(D^2,D.H)= (-1,1)$ and we are done.
\end{itemize}
The converse assertion is straightforward.
\end{proof}

Note that one can also obtain an explicit list of such divisors from Propisitions \ref{P37}, \ref{P32} and [\cite{PLT}, Theorem $4.7$].\

\section{Examples of $\ell$-away ACM line bundles }\label{ex} In this section, our  aim is to prove the third main theorem of this paper i.e. to produce explicit examples of $\ell$-away ACM line bundles on a nonsingular cubic surface $X$, where $\ell>2$. Before proceeding for that theorem, in the following example, we use the $2$-away classification to extend our discussion on $\ell$-away ACM line bundles of degree $1,2$ (in subsection $\S$ \ref{12}) to degree $3$ divisors on $X$.  To be more precise, we determine the values of $\ell$ for which a non-zero effective divisor $D$ on $X$ with $D.H=3$ becomes $\ell$-away ACM.\

\begin{example}\label{ex0}\normalfont
Let $D$ be a non-zero effective divisor on $X$ with $D.H=3$. Then there are five possibilities:
\begin{itemize}
	\item $D$ is a reduced and irreducible divisor on $X$ with $D.H=3$. It is easy to see that $D^2=1$. Therefore, by  Theorem \ref{Ton}, $(3)$, $\mathcal O_X(D)$ is ACM.\
	
	\item $D=D_1+Q$, where $D_1$ is line on $X$ and $Q$ is a reduced and irreducible divisor on $X$ with $Q.H=2$. If $D_1.Q=0$, then $(D^2, D.H)=(-1,3)$ and hence by Theorem \ref{T35}, $(b)(v)$ $\mathcal O_X(D)$ is $2$-away ACM. If $D_1.Q=1$, then $D^2=1$ and hence by  Theorem \ref{Ton}, $(3)$, $\mathcal O_X(D)$ is ACM. If $D_1.Q=2$, then $(D^2, D.H)=(3,3)$. Then applying \ref{RR} to the line bundle $\mathcal O_X(D-H)$, we get $h^0(\mathcal O_X(D-H)) \neq 0$, whence it follows that $\mathcal O_X(D) \cong \mathcal O_X(H)$ and hence $\mathcal O_X(D)$ is ACM but not initialized.\
	
	\item $D=D_1+2D_2$, where $D_i$'s are lines on $X$ for $i=1,2$ with $D_1.D_2 \geq 0$. If $D_1.D_2=1$, then $(D^2, D.H)=(-1,3)$ and hence by Theorem \ref{T35}, $(b)(v)$ $\mathcal O_X(D)$ is $2$-away ACM.\
	
	 If $D_1.D_2=0$, then $(D^2, D.H)=(-5,3)$ and we proceed as follows: applying \ref{RR} to the line bundles $\mathcal O_X(D-2H), \mathcal O_X(D-H)$ and $\mathcal O_X(D)$, we get $h^1(\mathcal O_X(D-2H)) \neq 0, h^1(\mathcal O_X(D-H)) \neq 0$ and  $h^1(\mathcal O_X(D)) \neq 0$ respectively (Note that for the last case, one get by \ref{RR} that $h^1(\mathcal O_X(D)) =h^0(\mathcal O_X(D))$ and the later is not zero can be achieved using standard S.E.S's).  Next, we show that $h^1(\mathcal O_X(D+tH)) =0$ for $t \in \mathbb Z \setminus \{-2,-1,0\}$. We divide this into two parts.\
	
	 Since $3H-D$ is nef and big, by Kawamata-Viehweg vanishing, we get $h^1(\mathcal O_X(D-3H))=0$. The fact that $h^1(\mathcal O_X(D-tH))=0$ for all $t \geq 3$ follows inductively by looking at the piece of the cohomology sequence,
	\begin{align*}
		\cdots \to H^0(\mathcal O_H(D-tH)) \to H^1(\mathcal O_X(D-(t+1)H)) \to H^1(\mathcal O_X(D-tH)) \to \cdots
	\end{align*}
	
	For vanishings in the other direction, we proceed as follows: by \ref{RR} applied to $\mathcal O_X(D+H)$, we get $h^1(\mathcal O_X(D+H))+6=h^0(\mathcal O_X(D+H))$. Tensorizing the S.E.S \ref{ss} (with $\Gamma_1$ replaced by $D_2$) by $\mathcal O_X(D+H)$ and taking cohomology, we get, $h^0(\mathcal O_X(D+H))=h^0(\mathcal O_X(D_1+D_2+H))$. Tensorizing the S.E.S \ref{ss} (with $\Gamma_1$ replaced by $D_1$) by $\mathcal O_X(D_1+D_2+H)$ and taking cohomology, we get, $h^0(\mathcal O_X(D_1+D_2+H)) \leq h^0(\mathcal O_X(D_2+H))+1$. Tensorizing the S.E.S \ref{ss} (with $\Gamma_1$ replaced by $H$) by $\mathcal O_X(D_2+H)$ and taking cohomology, we get, $h^0(\mathcal O_X(D_2+H))=5$. This forces $h^0(\mathcal O_X(D+H)) \leq 6$ and therefore, $h^1(\mathcal O_X(D+H))=0$. It is easy to see that $\mathcal O_X(D+2H)$ is $0$-regular and hence by Theorem \ref{CM}, $h^1(\mathcal O_X(D+tH))=0$ for all $t \geq 1$. This shows $\mathcal O_X(D)$ is $3$-away ACM. \

	\item $D=D_1+D_2+D_3$, where $D_i$'s are lines on $X$ for $i=1,2,3$ with $D_i.D_j \geq 0$ for $i \neq j \in \{1,2,3\}$.  If $D_i.D_j = 0$ for $i \neq j \in \{1,2,3\}$,  then $(D^2, D.H)=(-3,3)$ and hence by Theorem \ref{T35}, $(b)(iv)$ $\mathcal O_X(D)$ is $2$-away ACM. If without loss of generality, $D_1.D_2=1$ and other two intersection numbers are $0$, then $(D^2, D.H)=(-1,3)$ and hence by Theorem \ref{T35}, $(b)(v)$ $\mathcal O_X(D)$ is $2$-away ACM. If without loss of generality, $D_1.D_2=1, D_1.D_3=1$ and $D_2.D_3=0$, then $D^2=1$ and hence by Theorem \ref{Ton}, $(3)$, $\mathcal O_X(D)$ is ACM.  If $D_i.D_j =1$ for $i \neq j \in \{1,2,3\}$, then $(D^2, D.H)=(3,3)$. Then applying \ref{RR} to the line bundle $\mathcal O_X(D-H)$, we get $h^0(\mathcal O_X(D-H)) \neq 0$, whence it follows that $\mathcal O_X(D) \cong \mathcal O_X(H)$ and hence $\mathcal O_X(D)$ is ACM  but not initialized.\
	
	\item $D=3D_1$, where $D_1$ is a  line on $X$. Applying \ref{RR} to the line bundles $\mathcal O_X(D-3H), \mathcal O_X(D-2H), \mathcal O_X(D-H)$ and $\mathcal O_X(D)$, we get $h^1(\mathcal O_X(D-3H)) \neq 0, h^1(\mathcal O_X(D-2H)) \neq 0, h^1(\mathcal O_X(D-H)) \neq 0$ and $h^1(\mathcal O_X(D)) \neq 0$ respectively. Next, we see that by \ref{RR}, $\chi(\mathcal O_X(D+H) = 4$. Applying cohomology to the following S.E.S
	\begin{align*} 
	0 \rightarrow \mathcal{O}_X(2D_1+H) \rightarrow \mathcal{O}_X(3D_1+H) \rightarrow \mathcal{O}_{D_1}(3D_1+H) \rightarrow 0,
	\end{align*}
	 we get $H^0(\mathcal{O}_X(2D_1+H)) \cong H^0(\mathcal{O}_X(3D_1+H))$. Using similar S.E.S, we get $H^0(\mathcal{O}_X(D_1+H)) \cong H^0(\mathcal{O}_X(2D_1+H))$. Since by \ref{RR}, $\chi(\mathcal O_X(D_1+H))=5$, we get $h^0(\mathcal O_X(D_1+H)) \ge 5$. Hence, $h^1(\mathcal O_X(D+H)) \ge 1$. Next, we  show that $h^1(\mathcal O_X(D+tH)) =0$ for $t \in \mathbb Z \setminus \{-3,-2,-1,0,1\}$. We divide this into two parts.\
	 
	 Since $H-D_1$ is nef, this implies $4H-D=4H-3D_1=H+(3H-3D_1)$ is nef. Since $4H-D$ is also big, by Kawamata-Viehweg vanishing, we get $h^1(\mathcal O_X(D-4H))=0$. The fact that $h^1(\mathcal O_X(D-tH))=0$ for all $t \geq 4$ follows inductively by looking at the piece of the cohomology sequence,
	 \begin{align*}
	 \cdots \to H^0(\mathcal O_H(D-tH)) \to H^1(\mathcal O_X(D-(t+1)H)) \to H^1(\mathcal O_X(D-tH)) \to \cdots
	 \end{align*}
	 
	 Next, note that $3D_1+3H$ is nef and big. Therefore, by Kawamata-Viehweg vanishing, we have $H^1(\mathcal O_X(D+2H))=0$. It is easy to see that $\mathcal O_X(D+3H)$ is $0$-regular and hence by Theorem \ref{CM}, $h^1(\mathcal O_X(D+tH))=0$ for all $t \geq 2$. This shows $\mathcal O_X(D)$ is $5$-away ACM.\\
	 \end{itemize}	
\end{example}

We now move on to prove the third main theorem of this paper i.e. the examples of initialized and $\ell$-away ACM line bundles on a nonsingular cubic surface $X$ for all natural number $\ell$ but the odd multiples of $3$ i.e. for all $\ell \in \mathbb{N} \setminus 3(2\mathbb{N}+1)$. Recall that $X$ is isomorphic with $\mathbb{P}^2$ blown up at $6$  points $P_1, P_2, \cdots, P_6$ in general  position. Let $\pi:X \rightarrow \mathbb{P}^2$ be the blow up map. Then any divisor in $X$ can be written as $\alpha l + \sum_{i=1}^6\beta_ie_i$, where $\alpha, \beta_i$ are integers, and $l$ is the class of the pullback of a line in $\mathbb{P}^2$ and $e_i$ are $6$ exceptional divisors.\

\begin{theorem}\label{Ext}
	Let $a \ge 0$ be an integer. Then with the above notation, we will have the following:
	\begin{itemize}
		\item [(1)] The divisor $D=2l+a\sum_{i=1}^{6}e_{i}$ is $(6a+2)$-away ACM. The nonvanishing cohomologies are $H^{1}(X, \mathcal O_X(D-(5a+3)H)), \cdots, H^{1}(X, \mathcal O_X(D+(a-2)H))$. \\
		
		\item[(2)] The divisor $D=2l+(a+1)e_{1}+a\sum_{i=2}^{6}e_{i}$ is $(6a+4)$-away ACM. The nonvanishing cohomologies are $H^{1}(X, \mathcal O_X(D-(5a+4)H)), \cdots, H^{1}(X, \mathcal O_X(D+(a-1)H))$.\\

		\item [(3)] The divisor $D=2l+(a+1)e_{1}+(a+1)e_{2}+ a\sum_{i=3}^{6}e_{i}$ is $(6a+5)$-away ACM. The nonvanishing cohomologies are $H^{1}(X, \mathcal O_X(D-(5a+5)H)), \cdots, H^{1}(X, \mathcal O_X(D+(a-1)H))$.\\

		\item [(4)] The divisor $D=2l+(a+1)e_{1}+(a+1)e_{2}+(a+1)e_{3}+a\sum_{i=4}^{6}e_{i}$ is $(6a+6)$-away ACM. The nonvanishing cohomologies are $H^{1}(X, \mathcal O_X(D-(5a+6)H)), \cdots, H^{1}(X, \mathcal O_X(D+(a-1)H))$.  \\

		\item [(5)] The divisor $D=2l+(a+1)e_{1}+(a+1)e_{2}+(a+1)e_{3}+(a+1)e_{4}+a\sum_{i=5}^{6}e_{i}$ is $(6a+7)$-away ACM. The nonvanishing cohomologies are $H^{1}(X, \mathcal O_X(D-(5a+7)H)), \cdots, H^{1}(X, \mathcal O_X(D+(a-1)H))$. \\

		\item [(6)] The divisor $D=2l+(a+1)e_{1}+(a+1)e_{2}+(a+1)e_{3}+(a+1)e_{4}+(a+1)e_{5}+ae_{6}$ is $(6a+8)$-away ACM. The nonvanishing cohomologies are $H^{1}(X, \mathcal O_X(D-(5a+8)H)), \cdots, H^{1}(X, \mathcal O_X(D+(a-1)H))$. \\

	\end{itemize}
\end{theorem}
\begin{proof}
	The line bundles in all the six cases are initialized can be seen in the following way. Since all these divisors are effective, we only need to show $H^0(X, \mathcal O_X(D-H))=0$. In all these examples the divisor $D-H$ will be of the form $-l + \sum_{i=1}^{6}a_{i}e_{i}$ for $a_{i} \ge 0$.\\
	
	\textbf{\underline{Case1}:} Let us assume $a_{i} =0$ for all $i$. We have $H^{0}(X, \mathcal O_X(D-H)) = H^{0}(X, \mathcal O_X(-l)) = H^0(X, \pi^*\mathcal{O}_{\mathbb{P}^2}(-1))$. Using projection formula for the blow up map $\pi:X \rightarrow \mathbb{P}^{2}$, we have $R^{i}\pi_{*}(\pi^*\mathcal{O}_{\mathbb{P}^2}(-1)) = \mathcal{O}_{\mathbb{P}^2}(-1) \otimes R^{i}\pi_{*}\mathcal O_X =0$ for all $i >0$. Thus,  $H^0(X, \pi^*\mathcal{O}_{\mathbb{P}^2}(-1))= H^0(X, \pi_*\pi^*\mathcal{O}_{\mathbb{P}^2}(-1))=H^{0}(\mathbb{P}^2, \mathcal{O}_{\mathbb{P}^{2}}(-1)) =0$. \\
	
	\textbf{\underline{Inductive argument}:} Let us assume $a_{i} >0$ for some $i$. We consider the following S.E.S
	\begin{equation}
		\label{EqRHS}
		0 \rightarrow \mathcal{O}_{X}(D-H-e_i) \rightarrow \mathcal{O}_{X}(D-H) \rightarrow \mathcal{O}_{\mathbb{P}^{1}}(-a_{i}) \rightarrow 0
	\end{equation}
	Since $a_i>0$, we will have $H^0(\mathbb{P}^1, \mathcal{O}_{\mathbb{P}^{1}}(-a_{i}))=0$. Taking the long exact sequence for the above S.E.S we get
	\begin{equation}
		\label{EqRHS2}
		H^{0}(\mathcal O_X(D-H)) = H^0(X, \mathcal{O}_{X}(D-H-e_i))=H^{0}(\mathcal{O}_{X}(-l + (a_{i}-1)e_{i}+\sum_{j \ne i}a_{j}e_{j}))
	\end{equation}
	Thus, using the above, we can reduce the values of $a_i$ to $0$. Thus, without loss of generality, we can assume that $a_{i}=0$ for all $i$. Hence, the initializedness of the divisors follows from the case $1$.\\
	
	\textbf{\underline{The last nonvanishing cohomology on LHS}}: Let $D$ be a divisor on $X$ such that there exists a line $L$ on $X$ with $D \cdot L \le -2$. We consider the following S.E.S
	$$0 \rightarrow \mathcal{O}_X(D-L) \rightarrow \mathcal{O}_X(D) \rightarrow \mathcal{O}_L(D) \rightarrow 0$$
	Since $D \cdot L <0$, we will have $H^0(L, \mathcal{O}_L(D))=0$. This implies $H^0(X, \mathcal O_X(D-L)) \cong H^0(X, \mathcal O_X(D))$.\\
	
	
	 Note that $$\chi(\mathcal O_X(D-L))=\frac{(D-L)(D-L+H)}{2}+1=\frac{D(D+H)}{2}+1-(D \cdot L +1)=\chi(\mathcal O_X(D))-(D \cdot L +1).$$
	Since $D \cdot L \le -2$, we will have  
	\begin{equation}
		\label{LHSeq}
		\chi(\mathcal O_X(D-L)) \ge \chi(\mathcal O_X(D)) +1
	\end{equation}
	
	We will assume $H^2(X, \mathcal O_X(D-L))=0$, and $H^2(X, \mathcal O_X(D))=0$. Since we already have $H^0(X, \mathcal O_X(D-L)) \cong H^0(X, \mathcal O_X(D))$, from \eqref{LHSeq} we get $h^1(X, \mathcal O_X(D)) \ge 1$.\\
	
	We will check the following three assumptions 
	\begin{itemize}
		\item[(1)] $D \cdot L \le -2$
		\item[(2)]$H^2(\mathcal O_X(D))=0$
		\item[(3)]$H^2(\mathcal O_X(D-L))=0$
	\end{itemize}
	for the divisor on extreme LHS in each of the six cases to prove the nonvanishing of the cohomology.
	\begin{itemize}
		\item[(1)] $(5a+2)H-D=(15a+4)l-(6a+2)\sum_{i=1}^{6}e_i$. For the line $L=2l-\sum_{i=1}^{5}e_i$, we have $((5a+2)H-D) \cdot L=2(15a+4)-5(6a+2)=-2$.\\
		
		Also $H^2(X, \mathcal O_X((5a+2)H-D))=H^0(X, \mathcal O_X(D-(5a+3)H))=0$ since by initializedness, we have $H^0(X, \mathcal O_X(D-H))=0$. Also $H^2(X, \mathcal O_X((5a+2)H-D-L))=H^0(X, \mathcal O_X(D+L-(5a+3)H))=0$ as $ (D+L-(5a+3)H) \cdot H =2\cdot 3 + 6a+1-3(5a+3) <0$.\\
		
		Thus, we will have $h^1(X, \mathcal O_X((5a+2)H-D)) \ge 1$. By Serre duality, we have $h^1(X, \mathcal O_X(D-(5a+3)H))= h^1(X, \mathcal O_X((5a+2)H-D)) \ge 1$. \\
		
		The divisor $(5a+4)H-D(=2l+a\sum_{i=1}^{6}e_{i}) = (15a+10)H - (6a+4)\sum_{i=1}^{6}e_{i}$ can be checked to be nef by Theorem \eqref{N}. Since $((5a+4)H-D) \cdot ((5a+4)H-D) >0$, the divisor $(5a+4)H-D$ is big. Hence, by Kawamata-Viehweg vanishing, we get $H^{1}(X, \mathcal O_X(D-(5a+4)H))=0$.\\
		
		\item[(2)] $(5a+3)H-D=(15a+7)l-(6a+4)e_{1}-(6a+3)\sum_{i=2}^{6}e_{i}$. Thus, $((5a+3)H-D) \cdot L(=2l-\sum_{i=1}^{5}e_{i})=-2$. \\
		
		Also $D-(5a+4)H$ and $(D+L-(5a+4)H)$ intersects with $H$ negatively. Thus, $H^2(X, \mathcal O_X((5a+3)H-D))=H^0(X, \mathcal O_X(D-(5a+4)H))=0$ and  $H^2(X, \mathcal O_X((5a+3)H-D-L))=H^0(X, \mathcal O_X(D+L-(5a+4)H))=0$.\\
		
		Hence, $h^1(X, \mathcal O_X((5a+3)H-D)) \ge 1$. By Serre duality, we have $h^1(X, \mathcal O_X(D-(5a+4)H)) \ge 1$.\\
		
		The divisor $(5a+5)H-D$ can be checked to be both nef and big by similar arguments as in $(1)$. Thus by Kawamata-Viehweg vanishing, $H^{1}(X, \mathcal O_X(D-(5a+5)H))=0$.\\
		
		\item[(3)] $(5a+4)H-D=(15a+10)l-(6a+5)e_{1}-(6a+5)e_{2}-(6a+4)\sum_{i=3}^{6}e_{i}$. Thus, $((5a+4)H-D) \cdot L(=2l-\sum_{i=1}^{5}e_{i})=-2$. \\
		
		Also $D-(5a+5)H$ and $(D+L-(5a+5)H)$ intersects with $H$ negatively. Thus, $H^2(X, \mathcal O_X((5a+4)H-D))=H^0(X, \mathcal O_X(D-(5a+5)H))=0$ and  $H^2(X, \mathcal O_X((5a+4)H-D-L))=H^0(X, \mathcal O_X(D+L-(5a+5)H))=0$. Hence, by Serre duality $h^1(X, \mathcal O_X(D-(5a+5)H)) \ge 1$. \\
		
		The divisor $(5a+6)H-D$ is nef and big. Thus, by Kawamata-Viehweg vanishing, $H^{1}(X, \mathcal O_X(D-(5a+6)H))=0$.\\
		
		\item[(4)] $(5a+5)H-D=(15a+13)l-(6a+6)e_{1}-(6a+6)e_{2}-(6a+6)e_{3}-(6a+5)\sum_{i=4}^{6}e_{i}$. Thus, $((5a+5)H-D) \cdot L(=2l-\sum_{i=1}^{5}e_{i})=-2$. \\
		
		$D-(5a+6)H$ and $(D+L-(5a+6)H)$ intersects with $H$ negatively. Thus, $H^2(\mathcal O_X((5a+5)H-D))=H^0(\mathcal O_X(D-(5a+6)H))=0$ and $H^2(\mathcal O_X((5a+5)H-D-L))=H^0(\mathcal O_X(D+L-(5a+6)H))=0$. Hence, by Serre duality, $h^1(X, \mathcal O_X(D-(5a+6)H)) \ge 1$.\\
		
		The divisor $(5a+7)H-D$ is nef and big. Thus, by Kawamata-Viehweg vanishing, $H^{1}(\mathcal O_X(D-(5a+7)H))=0$.\\
		
		\item[(5)] $(5a+6)H-D=(15a+16)l-(6a+7)e_{1}-(6a+7)e_{2}-(6a+7)e_{3}-(6a+7)e_{4}-(6a+6)\sum_{i=5}^{6}e_{i}$. Thus, $((5a+6)H-D) \cdot L(=2l-\sum_{i=1}^{5}e_{i})=-2$. \\
		
		 $D-(5a+7)H$ and $(D+L-(5a+7)H)$ intersects with $H$ negatively. Thus, $H^2(\mathcal O_X((5a+6)H-D))=H^0(\mathcal O_X(D-(5a+7)H))=0$ and $H^2(\mathcal O_X((5a+6)H-D-L))=H^0(\mathcal O_X(D+L-(5a+7)H))=0$. Hence, by Serre duality, $h^1(X, \mathcal O_X(D-(5a+7)H)) \ge 1$.\\
		
		The divisor $(5a+8)H-D$ is nef and big. Thus, by Kawamata-Viehweg vanishing, $H^{1}(\mathcal O_X(D-(5a+8)H))=0$.\\
		
		\item[(6)] $(5a+7)H-D=(15a+19)l-(6a+8)e_{1}-(6a+8)e_{2}-(6a+8)e_{3}-(6a+8)e_{4}-(6a+8)e_{5}-(6a+7)e_{6}$. Thus, $((5a+7)H-D) \cdot L(=2l-\sum_{i=1}^{5}e_{i})=-2$. \\
		
		 $D-(5a+8)H$ and $(D+L-(5a+8)H)$ intersects with $H$ negatively. Thus, $H^2(\mathcal O_X((5a+7)H-D))=H^0(\mathcal O_X(D-(5a+8)H))=0$ and $H^2(\mathcal O_X((5a+7)H-D-L))=H^0(\mathcal O_X(D+L-(5a+8)H))=0$. Hence, by Serre duality, $h^1(X, \mathcal O_X(D-(5a+8)H)) \ge 1$.\\
		
		The divisor $(5a+9)H-D$ is nef and big. Thus, by Kawamata-Viehweg vanishing, $H^{1}(\mathcal O_X(D-(5a+9)H))=0$.\
	\end{itemize}

	\textbf{\underline{The last nonvanishing on the RHS}}: Let $D'$ be the divisor $D+(a-2)H$ in $(1)$ and $D+(a-1)H$ for $(2)$, $(3)$, $\cdots$, $(6)$. In each of these cases we will compute $H^{0}(X, \mathcal O_X(D'))$ and show that $h^{0}(X, \mathcal O_X(D')) > \chi(\mathcal O_X(D'))$. Thus, we will have $h^{1}(X, \mathcal O_X(D')) \ne 0$. Further, we will show that $h^{0}(X, \mathcal O_X(D'+H)) = \chi(\mathcal O_X(D'+H))$ and $H^{2}(X, \mathcal O_X(D'+H))=0$. This will imply $H^{1}(X, \mathcal O_X(D'+H)) = 0$. We will check these computations in each of the six cases.\\
	
	\item [(1)] $D+(a-2)H=(3(a-2)+2)l + 2\sum_{i=1}^{6}e_{i}=(3a-4) l + 2\sum_{i=1}^{6}e_{i}$. Using the arguments used in the \textbf{\underline{Inductive argument}} \eqref{EqRHS} and \eqref{EqRHS2}, it can proved that $H^{0}(X, \mathcal O_X(D+(a-2)H))=H^{0}(X, \pi^*\mathcal{O}_{\mathbb{P}^2}(3a-4))$. By projection formula, we have $R^i\pi_*\pi^*\mathcal{O}_{\mathbb{P}^2}(3a-4)=\mathcal{O}_{\mathbb{P}^2}(3a-4) \otimes R^i \pi_* \mathcal{O}_X=0$ for all $i>0$. Thus, we have $H^{0}(X, \pi^*\mathcal{O}_{\mathbb{P}^2}(3a-4))=H^{0}(X, \pi_*\pi^*\mathcal{O}_{\mathbb{P}^2}(3a-4))=H^{0}(\mathbb{P}^{2}, \mathcal{O}_{\mathbb{P}^{2}}(3a-4))$. The dimension of $H^{0}(\mathbb{P}^{2}, \mathcal{O}_{\mathbb{P}^{2}}(3a-4))$ is $\left( \begin{array}{c} 3a-4+2 \\ 2 \end{array} \right) =\frac{(3a-2)(3a-3)}{2}=\frac{9a^2-15a+6}{2}$.\\
	
	On the other hand, $\chi(\mathcal O_X(D+(a-2)H))=\frac{((3a-4) l + 2\sum_{i=1}^{6}e_{i})((3a-1) l + \sum_{i=1}^{6}e_{i})}{2}+1=\frac{9a^2-15a-6}{2}$. This implies that $h^1(\mathcal O_X(D+(a-2)H)) \ge 6$.\\
	
	$D+(a-1)H=(3a-1) l + \sum_{i=1}^{6}e_{i}$. Thus using S.E.S's similar to \eqref{EqRHS} and \eqref{EqRHS2}, we get  $h^{0}(X, \mathcal O_X(D+(a-1)H))=h^{0}(\mathbb{P}^{2}, \mathcal{O}_{\mathbb{P}^{2}}(3a-1))=\left( \begin{array}{c} 3a-1+2 \\ 2 \end{array} \right) =\frac{(3a+1)(3a)}{2}$.\\
	
	On the other hand, $\chi(\mathcal O_X(D+(a-1)H))=\frac{((3a-1) l + \sum_{i=1}^{6}e_{i})((3a+2) l )}{2}+1=\frac{(3a-1)(3a+2)+2}{2}=\frac{9a^2+3a}{2}=\frac{3a(3a+1)}{2}$. Also $H^2(X, \mathcal O_X(D+(a-1)H))=H^0(X, \mathcal O_X(-D-aH))=0$ as $-D-aH$ intersects $H$ negatively. This implies $H^1(X, \mathcal O_X(D+(a-1)H))=0$.\\
	
	\item [(2)]  $D+(a-1)H=(3(a-1)+2)l + 2e_{1}+\sum_{i=2}^{6}e_{i}=(3a-1) l + 2e_{1}+\sum_{i=2}^{6}e_{i}$. Thus, we get $h^{0}(X, \mathcal O_X(D+(a-1)H))=\text{dim}H^{0}(\mathbb{P}^{2}, \mathcal{O}_{\mathbb{P}^{2}}(3a-1))=\left( \begin{array}{c} 3a-1+2 \\ 2 \end{array} \right) =\frac{(3a+1)(3a)}{2}=\frac{9a^2+3a}{2}$.\\
	
	On the other hand, $\chi(\mathcal O_X(D+(a-1)H))=\frac{((3a-1) l + 2e_{1}+\sum_{i=2}^{6}e_{i})((3a+2) l + e_{1}}{2}+1=\frac{(3a-1)(3a+2)-2}{2}+1$. This implies $h^1(X, \mathcal O_X(D+(a-1)H)) \ge 1$.\\
	
	$D+aH=(3a+2)l + e_{1}$. Thus, we get $h^{0}(X, \mathcal O_X(D+aH))=h^{0}(\mathbb{P}^{2}, \mathcal{O}_{\mathbb{P}^{2}}(3a+2))=\left( \begin{array}{c} 3a+2+2 \\ 2 \end{array} \right) =\frac{(3a+4)(3a+3)}{2}$.\\
	
	On the other hand, $\chi(\mathcal O_X(D+aH))=\frac{((3a+2) l + e_{1})((3a+5) l -\sum_{i=2}^{6}e_{i})}{2}+1=\frac{(3a+5)(3a+2)}{2}+1=\frac{9a^2+21a+12}{2}$. Also $H^2(X, \mathcal O_X(D+aH))=H^0(X, \mathcal O_X(-D-(a+1)H))=0$. This implies $H^1(X, \mathcal O_X(D+aH))=0$.\\
	
	\item [(3)] Let $D=2l+(a+1)e_{1}+(a+1)e_{2}+\sum_{i=3}^{6}e_{i}$. Then $D+(a-1)H=(3a-1) l + 2e_{1}+2e_{2}+\sum_{i=3}^{6}e_{i}$. Thus, we get $h^{0}(X, \mathcal O_X(D+(a-1)H))=h^{0}(\mathbb{P}^{2}, \mathcal{O}_{\mathbb{P}^{2}}(3a-1))=\left( \begin{array}{c} 3a-1+2 \\ 2 \end{array} \right) =\frac{(3a+1)(3a)}{2}=\frac{9a^2+3a}{2}$.\\
	
	On the other hand, $\chi(\mathcal O_X(D+(a-1)H))=\frac{((3a-1) l + 2e_{1}+2e_{2}+\sum_{i=3}^{6}e_{i})((3a+2) l + e_{1}+e_{2}}{2}+1=\frac{(3a-1)(3a+2)-2-2}{2}+1=\frac{9a^{2}+3a-4}{2}$. This implies $h^1(X, \mathcal O_X(D+(a-1)H)) \ge 2$.\\
	
	$D+aH=(3a+2)l + e_{1}+e_{2}$. Thus, we get $h^{0}(X, \mathcal O_X(D+aH))=h^{0}(\mathbb{P}^{2}, \mathcal{O}_{\mathbb{P}^{2}}(3a+2))=\left( \begin{array}{c} 3a+2+2 \\ 2 \end{array} \right) =\frac{(3a+4)(3a+3)}{2}$.\\
	
	On the other hand, $\chi(\mathcal O_X(D+aH))=\frac{((3a+2) l + e_{1}+e_{2})((3a+5) l -\sum_{i=3}^{6}e_{i})}{2}+1=\frac{(3a+5)(3a+2)}{2}+1=\frac{9a^2+21a+12}{2}$. We have $H^2(\mathcal O_X(D+aH))=H^0(\mathcal O_X(-D-(a+1)H))=0$. This implies $H^1(\mathcal O_X(D+aH))=0$.\\
	
	\item [(4)] Let $D=2l+(a+1)e_{1}+(a+1)e_{2}+(a+1)e_{3}+a\sum_{i=4}^{6}e_{i}$. Then  $D+(a-1)H=(3a-1) l + 2e_{1}+2e_{2}+2e_{3}+\sum_{i=4}^{6}e_{i}$. Thus, we have $h^{0}(X, \mathcal O_X(D+(a-1)H))=h^{0}(\mathbb{P}^{2}, \mathcal{O}_{\mathbb{P}^{2}}(3a-1))=\left( \begin{array}{c} 3a-1+2 \\ 2 \end{array} \right) =\frac{(3a+1)(3a)}{2}=\frac{9a^2+3a}{2}$.\\
	
	On the other hand, $\chi(\mathcal O_X(D+(a-1)H))=\frac{((3a-1) l + 2e_{1}+2e_{2}+2e_{3}+\sum_{i=4}^{6}e_{i})((3a+2) l + e_{1}+e_{2}+e_{3})}{2}+1=\frac{(3a-1)(3a+2)-2-2-2}{2}+1=\frac{9a^{2}+3a-6}{2}$. This implies that $h^1(X, \mathcal O_X(D+(a-1)H)) \ge 3$.\\
	
	$D+aH=(3a+2)l + e_{1}+e_{2}$. Thus, we have $h^{0}(X, \mathcal O_X(D+aH))=h^{0}(\mathbb{P}^{2}, \mathcal{O}_{\mathbb{P}^{2}}(3a+2))=\left( \begin{array}{c} 3a+2+2 \\ 2 \end{array} \right) =\frac{(3a+4)(3a+3)}{2}$.\\
	
	On the other hand, $\chi(\mathcal O_X(D+aH))=\frac{((3a+2) l + e_{1}+e_{2})((3a+5) l -\sum_{i=3}^{6}e_{i})}{2}+1=\frac{(3a+5)(3a+2)}{2}+1=\frac{9a^2+21a+12}{2}$. We also have $H^2(\mathcal O_X(D+aH))=H^0(\mathcal O_X(-D-(a+1)H))=0$. This implies $H^1(\mathcal O_X(D+aH))=0$.\\
	
	\item [(5)] Let $D=2l+(a+1)e_{1}+(a+1)e_{2}+(a+1)e_{3}+(a+1)e_{4}+a\sum_{i=5}^{6}e_{i}$. Then $D+(a-1)H=(3a-1) l + 2e_{1}+2e_{2}+2e_{3}+2e_{4}+\sum_{i=5}^{6}e_{i}$. Thus, $h^{0}(X, \mathcal O_X(D+(a-1)H))=h^{0}(\mathbb{P}^{2}, \mathcal{O}_{\mathbb{P}^{2}}(3a-1))=\left( \begin{array}{c} 3a-1+2 \\ 2 \end{array} \right) =\frac{(3a+1)(3a)}{2}=\frac{9a^2+3a}{2}$.\\
	
	On the other hand, $\chi(\mathcal O_X(D+(a-1)H))=\frac{((3a-1) l + 2e_{1}+2e_{2}+2e_{3}+2e_{4}+\sum_{i=5}^{6}e_{i})((3a+2) l + e_{1}+e_{2}+e_{3}+e_{4})}{2}+1=\frac{(3a-1)(3a+2)-2-2-2-2}{2}+1=\frac{9a^{2}+3a-8}{2}$. This implies $h^1(X, \mathcal O_X(D+(a-1)H)) \ge 4$.\\
	
	$D+aH=(3a+2)l + e_{1}+e_{2}+e_{3}+e_{4}$. Thus, $h^{0}(X, \mathcal O_X(D+aH))=h^{0}(\mathbb{P}^{2}, \mathcal{O}_{\mathbb{P}^{2}}(3a+2))=\left( \begin{array}{c} 3a+2+2 \\ 2 \end{array} \right) =\frac{(3a+4)(3a+3)}{2}$.\\
	
	On the other hand, $\chi(\mathcal O_X(D+aH))=\frac{((3a+2) l + e_{1}+e_{2}+e_{3}+e_{4})((3a+5) l -\sum_{i=5}^{6}e_{i})}{2}+1=\frac{(3a+5)(3a+2)}{2}+1=\frac{9a^2+21a+12}{2}$. Also we have  $H^2(X, \mathcal O_X(D+aH))=H^0(X, \mathcal O_X(-D-(a+1)H))=0$. This implies $H^1(X, \mathcal O_X(D+aH))=0$.\\
	
	\item [(6)] Let $D=2l+(a+1)e_{1}+(a+1)e_{2}+(a+1)e_{3}+(a+1)e_{4}+(a+1)e_{5}+ae_{6}$. Then $D+(a-1)H=(3a-1) l + 2e_{1}+2e_{2}+2e_{3}+2e_{4}+2e_{5}+e_{6}$. Thus, we have $h^{0}(X, \mathcal O_X(D+(a-1)H))=h^{0}(\mathbb{P}^{2}, \mathcal{O}_{\mathbb{P}^{2}}(3a-1))=\left( \begin{array}{c} 3a-1+2 \\ 2 \end{array} \right) =\frac{(3a+1)(3a)}{2}=\frac{9a^2+3a}{2}$.\\
	
	On the other hand, $\chi(\mathcal O_X(D+(a-1)H))=\frac{((3a-1) l + 2e_{1}+2e_{2}+2e_{3}+2e_{4}+2e_{5}+e_{6})((3a+2) l + e_{1}+e_{2}+e_{3}+e_{4}+e_{5})}{2}+1=\frac{(3a-1)(3a+2)-2-2-2-2-2}{2}+1=\frac{9a^{2}+3a-10}{2}$. This implies $h^1(\mathcal O_X(D+(a-1)H)) \ge 5$.\\
	
	$D+aH=(3a+2)l + e_{1}+e_{2}+e_{3}+e_{4}+e_{5}$. Thus, we have  $h^{0}(\mathcal O_X(D+aH))=h^{0}(\mathcal{O}_{\mathbb{P}^{2}}(3a+2))=\left( \begin{array}{c} 3a+2+2 \\ 2 \end{array} \right) =\frac{(3a+4)(3a+3)}{2}$.\\
	
	On the other hand, $\chi(\mathcal O_X(D+aH))=\frac{((3a+2) l + e_{1}+e_{2}+e_{3}+e_{4})((3a+5) l -e_{6})}{2}+1=\frac{(3a+5)(3a+2)}{2}+1=\frac{9a^2+21a+12}{2}$. Also $H^2(\mathcal O_X(D+aH))=H^0(\mathcal O_X(-D-(a+1)H))=0$. This implies $H^1(\mathcal O_X(D+aH))=0$.\\
	
	\textbf{\underline{Nonvanishing of the cohomology of the intermediate twists}:} This follows as a corollary of Lemma \eqref{general}. For each of the six cases of the Theorem \eqref{Ext}, we have identified  the last nonvanishing terms on both the left and the right hand side. Let us denote them as $H^1(\mathcal O_X(D+m_1H)) \ne 0$ and $H^1(\mathcal O_X(D+m_2H)) \ne 0$, where $m_1 < m_2$. Let $m$ be an integer such that $m_1 < m <m_2$. If $H^1(\mathcal O_X(D+mH))=0$, then by Lemma \eqref{general}, depending on whether $\text{deg}(\mathcal{O}_X(D+mH)|_{H}) < 0$ or $\ge 0$, we will have either $H^1(X, \mathcal O_X(D+mH-nH))=0$ or $H^1(X, \mathcal O_X(D+mH+nH))=0$ for all $n >0$. This will imply that  either $H^1(\mathcal O_X(D+m_1H)) = 0$ or $H^1(\mathcal O_X(D+m_2H)) = 0$, which is a contradiction.\

	

	

	


	
\end{proof}

\begin{remark}
	\normalfont Let $D$ be a non-zero effective divisor on a nonsingular cubic surface $X$ such that $D$ is big but not nef. Thus, $D \cdot L <0$ for a line $L$ in $X$. Further, assume that $\mathcal O_X(D)$ is a $\ell$-away ACM line bundle. Then can we determine the value $m$ such that $\mathcal O_X(D-L)$ is $m$-away ACM? Is it also possible to express $m$ as a function of $\ell$, $D^2$, and $D \cdot L$?
\end{remark}


\section{Existence of $\ell$-away ACM line bundles on higher  degree  hypersurfaces}\label{hyp}



In this final section, we prove the last main theorem of this paper i.e. the existence of smooth hypersurfaces of any degree $d \geq 3$ admitting $\ell$-away ACM line bundles.\

\begin{theorem}\label{T41}
Let $\ell \geq 1$ be an integer. Then for any $d>\ell$, there exists a smooth hypersurface $X^{(d)} \subset \mathbb P^3$ of degree $d$  admitting  initialized and $\ell$-away ACM line bundles.

\end{theorem}

\begin{proof}
We divide the proof of this theorem into two steps:

\textbf{\underline{Step-1}}: We first prove the following claim which can also be thought of as an analogous version of [\cite{Watanabe}, Lemma $5.1$], [\cite{Debu}, Lemma $4.2$],   in the non-ACM situation. Our claim is as follows:

\textbf{\underline{Claim}} : Let $D$ be a reduced curve on a smooth quadric hypersurface $X^{(2)} \subset \mathbb P^3$ such that $\mathcal O_{X^{(2)}}(D)$ is a $\ell$-away ACM line bundle on $X^{(2)}$. Let $d>2$ be an integer such that :\

$(i)$ $h^0(\mathcal O_{X^{(2)}}(D)(1-d))=0$,\

$(ii)$ $h^1(\mathcal O_{X^{(2)}}(D)(-1-d))=0$.\

Then there exists a smooth hypersurface $X^{(d)} \subset \mathbb P^3$ of degree $d$ containing $D$ such that $\mathcal O_{X^{(d)}}(D)$ is an initialized and $\ell$-away ACM line bundle on $X^{(d)}$.\

\textbf{\underline{Proof of the Claim}} : Let $D$ be a reduced curve on a smooth quadric hypersurface $X^{(2)} \subset \mathbb P^3$ such that $\mathcal O_{X^{(2)}}(D)$ is a $\ell$-away ACM line bundle on $X^{(2)}$. Let $d>2$ be an integer such that conditions $(i)$ and $(ii)$ are satisfied. If we can show that $\mathcal I_D(d)$ is globally generated, then by [\cite{HMR} Proposition $5.1$ and Remark $5.2$]\footnote{Also see [\cite{Watanabe2}, Theorem $4.1$, Theorem $4.2$.]},  there exists a smooth hypersurface $X^{(d)} \subset \mathbb P^3$ of degree $d$ containing $D$. Thus by Castelnuovo-Mumford regularity criteria, it is enough to show that $\mathcal I_D(d)$ is $0$-regular i.e. $h^1(\mathcal I_D(d-1))=0, h^2(\mathcal I_D(d-2))=0$ and $h^3(\mathcal I_D(d-3))=0$. This can be achieved as follows: Since $D \subset X^{(2)}$, one can consider the following S.E.S
\begin{align}\label{41}
	0 \to \mathcal O_{\mathbb P^3}(-2) \to \mathcal I_D \to \mathcal O_{X^{(2)}}(-D) \to 0.
\end{align}
Tensorizing the S.E.S \ref{41} by $\mathcal O_{\mathbb P^3}(d-1)$, taking cohomology and using assumption $(ii)$, one obtains $h^1(\mathcal I_D(d-1))=0$. Tensorizing the S.E.S \ref{41} by $\mathcal O_{\mathbb P^3}(d-2)$, taking cohomology and using assumption $(i)$, one obtains $h^2(\mathcal I_D(d-2))=0$. Finally, Tensorizing the S.E.S \ref{41} by $\mathcal O_{\mathbb P^3}(d-3)$ and taking cohomology, one obtains $h^3(\mathcal I_D(d-3))=0$.\

Next, we  show that $\mathcal O_{X^{(d)}}(D)$ is $\ell$-away ACM. By assumption, $\mathcal O_{X^{(2)}}(D)$ is $\ell$-away ACM on $X^{(2)}$ and hence so is $\mathcal O_{X^{(2)}}(-D)$. Tensorizing the S.E.S \ref{41} by $\mathcal O_{\mathbb P^3}(m)$ and taking cohomology, we see that $h^1(\mathcal I_D(m))= h^1( \mathcal O_{X^{(2)}}(-D)(m)) \neq 0$ for exactly $\ell$ many integers $m$. Since $D \subset X^{(d)}$, one can consider the following S.E.S
\begin{align}\label{42}
	0 \to \mathcal O_{\mathbb P^3}(-d) \to \mathcal I_D \to \mathcal O_{X^{(d)}}(-D) \to 0.
\end{align}
Tensorizing the S.E.S \ref{42} by $\mathcal O_{\mathbb P^3}(m)$ and taking cohomology,  we see that $h^1( \mathcal O_{X^{(d)}}(-D)(m))=h^1(\mathcal I_D(m)) \neq0$ for exactly $\ell$ many integers $m$. This means $\mathcal O_{X^{(d)}}(-D)$ is $\ell$-away ACM on $X^{(d)}$ and hence so is $\mathcal O_{X^{(d)}}(D)$.\

Next, we  show that $\mathcal O_{X^{(d)}}(D)$ is initialized. Towards this, we first show that  $\mathcal O_{X^{(d)}}(D)$  is effective. Note that by Serre duality it is equivalent to show that $h^2(\mathcal O_{X^{(d)}}(-D)(d-4) ) \neq 0$. Tensorizing the S.E.S \ref{42} by $\mathcal O_{\mathbb P^3}(d-4)$ and taking cohomology, we see that $h^2(\mathcal O_{X^{(d)}}(-D)(d-4) ) + h^3(\mathcal I_D(d-4)) \geq 1$. Tensorizing the S.E.S \ref{41} $\mathcal O_{\mathbb P^3}(d-4)$ and taking cohomology, we have $h^3(\mathcal I_D(d-4))=0$ and we are done.\

Now we are only left to show that $h^0(\mathcal O_{X^{(d)}}(D)(-1) ) = 0$. Note that by Serre duality, this is equivalent to show that $h^2(\mathcal O_{X^{(d)}}(-D)(d-3) ) = 0$. Tensorizing the S.E.S \ref{42} by $\mathcal O_{\mathbb P^3}(d-3)$ and taking cohomology, we see that $h^2(\mathcal O_{X^{(d)}}(-D)(d-3)) = h^2(\mathcal I_D(d-3))$. Then tensorizing the S.E.S \ref{41} $\mathcal O_{\mathbb P^3}(d-3)$, taking cohomology and using the assumption $(i)$, we have $h^2(\mathcal I_D(d-3))=0$ and we are through.\

\textbf{\underline{Step-2}}: Now that we have established the claim in step-$1$, we aim to find a reduced curve $D$ on a smooth quadric hypersurface $X^{(2)} \subset \mathbb P^3$ such that $\mathcal O_{X^{(2)}}(D)$ is a $\ell$-away ACM line bundle on $X^{(2)}$ and satisfies the conditions of the above claim. Identify $X^{(2)} \cong \mathbb P^1 \times \mathbb P^1$ with the fundamental divisor $h=h_1+h_2$ as in [\cite{GG}, section $4$]. Then by [\cite{GG}, Theorem $4.1$], for each $\ell \geq 1$, $\mathcal O_{X^{(2)}}((l+1)h_1), \mathcal O_{X^{(2)}}((l+1)h_2)$ are the only initialized and $\ell$-away ACM line bundles on $X^{(2)}$. Let $D \in |\mathcal O_{X^{(2)}}((\ell+1)h_i)|, i \in \{1,2\}$ be a reduced curve. Then $\mathcal O_{X^{(2)}}(D)$ is initialized and $\ell$-away ACM on $X^{(2)}$.  Note that if $d>2$ and $ d > \ell$, then by K\"{u}nneth's theorem, $D$ satisfies the conditions $(i)$ and $(ii)$ of the claim. Therefore, for a given $\ell \geq 1$, if $d>2$ and $d>\ell$, then  there exists a smooth hypersurface $X^{(d)} \subset \mathbb P^3$ of degree $d$ containing a divisor $D$ such that $\mathcal O_{X^{(d)}}(D)$ is an initialized and $\ell$-away ACM line bundle on $X^{(d)}$.\

\end{proof}


\begin{thebibliography}{111}
		
			\bibitem{Debu} Bhattacharya, D. (2022). On initialized and ACM line bundles over a smooth sextic surface in $\mathbb P^3$. Communications in Algebra, 50(12), 5314-5344.\
		
		\bibitem{BPP} Bhattacharya, D., Parameswaran, A.J., \& Pine, J (2024), On $\ell$-away ACM bundles of higher rank on $\mathbb P^2$ (manuscript in preparation).\
		
		\bibitem{Beauville} Beauville, A. (2018). An introduction to Ulrich bundles. European journal of mathematics, 4(1), 26-36.\
		
		\bibitem{Chindea} Chindea, F. (2020). ACM line bundles on elliptic ruled surfaces. manuscripta mathematica, 161(1), 213-222.\
		
		\bibitem{Coskun} Coskun, E. (2017). A survey of Ulrich bundles. Analytic and algebraic geometry, 85-106.\
		
		\bibitem{book} Costa, L., Miró-Roig, R. M., \& Pons-Llopis, J. (2021). Ulrich bundles: from commutative algebra to algebraic geometry (Vol. 77). Walter de Gruyter GmbH \& Co KG.\
		
			\bibitem{96} di Rocco, S. (1996). k—Very Ample Line Bundles on Del Pezzo Surfaces. Mathematische Nachrichten, 179(1), 47-56.\
			
			\bibitem{Eisenbud} Eisenbud, D., Schreyer, F. O., \& Weyman, J. (2003). Resultants and Chow forms via exterior syzygies. Journal of the American Mathematical Society, 16(3), 537-579.\
		
			\bibitem{F} Faenzi, D. (2008). Rank 2 arithmetically Cohen–Macaulay bundles on a nonsingular cubic surface. Journal of Algebra, 319(1), 143-186.\
			
			\bibitem{GG} Gawron, F., \& Genc, O. (2023). $\ell$-away ACM bundles on Fano surfaces. Bollettino dell'Unione Matematica Italiana, 1-26.\
			
			\bibitem{88} Giuffrida, S. (1988). The Hilbert function of a curve lying on a smooth cubic surface. Annali di Matematica Pura ed Applicata (1923-), 153, 275-292.\
			
			\bibitem{90} Giuffrida, S., \& Maggioni, R. (1990). On the Rao module of a curve lying on a smooth cubic surface in $\mathbb P^3$. Communications in algebra, 18(7), 2039-2061.\
			
		
			\bibitem{92ii} Giuffrida, S., \& Maggioni, R. (1992). On the Rao module of a curve lying on a smooth cubic surface in $\mathbb P^3$, II. Communications in algebra, 20(2), 329-347.\
			
			\bibitem{92} Giuffrida, S., \& Maggioni, R. (1992). On the resolution of a curve lying on a smooth cubic surface in $\mathbb P^3$. Transactions of the American Mathematical Society, 331(1), 181-201.\
			
			\bibitem{H} Hartshorne, R. (2013). Algebraic geometry (Vol. 52). Springer Science \& Business Media.\
			
		\bibitem{k}	Knörrer, H. (1987). Cohen–Macaulay modules on hypersurface singularities I. Invent. Math. 88(1), 153–164.\
		
		
			
			
			\bibitem{quadric} Knörrer, H. (1986). Cohen-Macaulay modules on hypersurface singularities. Representations of Algebras (Durham, 1985), London Math. Soc. Lecture Notes, 116, 147-164.
			
				\bibitem{HMR} Hartshorne, R., \& Miró-Roig, R. M. (2015). On the intersection of ACM curves in $\mathbb P^3$. Journal of Pure and Applied Algebra, 219(8), 3195-3213.\
				
			\bibitem{PAG} 	Lazarsfeld, R. K. (2017). Positivity in algebraic geometry I: Classical setting: line bundles and linear series (Vol. 48). Springer.\
				
				\bibitem{okonek} Okonek, C., Schneider,  M., \&  Spindler. H. Vector bundles on complex projective spaces with an appendix by SI Gelfand. Progress in Math 3.\
				
				\bibitem{story} Ottaviani, G. (2024). Vector bundles without intermediate cohomology and the trichotomy result. Rendiconti del Circolo Matematico di Palermo Series 2, 1-15.\
				
				
				\bibitem{PLT} Pons-Llopis, J., \& Tonini, F. (2009). ACM Bundles on Del Pezzo surfaces. Le Matematiche, 64(2), 177-211.\
				
				
				\bibitem{K3} Watanabe, K. (2019). ACM line bundles on polarized K3 surfaces. Geometriae Dedicata, 203, 321-335.\
				
			
				
				\bibitem{quartic} Watanabe, K. (2015). The classification of ACM line bundles on quartic hypersurfaces in $\mathbb P^3$. Geometriae Dedicata, 175, 347-354.\
				
			
				
		
		\bibitem{Watanabe} Watanabe, K. (2021, October). The characterization of aCM line bundles on quintic hypersurfaces in $\mathbb P^3$. In Abhandlungen aus dem Mathematischen Seminar der Universität Hamburg (Vol. 91, pp. 179-197). Springer Berlin Heidelberg.\
		
		\bibitem{Watanabe2} Watanabe, K. (2023). On the classification of non-aCM curves on quintic surfaces in $\mathbb P^3$. Beiträge zur Algebra und Geometrie/Contributions to Algebra and Geometry, 1-20.\
		
	
		
	
		
	
		
		
		
		
		
		
		
		
		
		
		
		
		
		
		
		
		
		
		
		
		
		
		
		
	\end{thebibliography}
\end{document}